\documentclass[11pt]{article} 

\usepackage{fullpage}
\usepackage{amsthm}
\usepackage{amssymb}
\usepackage{amsmath}
\usepackage[mathscr]{eucal}
\usepackage{graphicx}
\usepackage{centernot}
\usepackage{textcomp}
\usepackage{afterpage}
\usepackage{esint}
\usepackage{xspace}
\usepackage{siunitx}
\usepackage{tikz}
\usepackage{hyperref}
\usepackage{bm}
\usepackage{wasysym}
\usepackage{accents}
\usepackage{xcolor}
\usepackage{setspace}

\usepackage[T1]{fontenc}
\usepackage[utf8]{inputenc}
\usepackage{authblk}
\usepackage[title]{appendix}

\newcommand{\abs}[1]{\left| #1 \right|}

\newcommand{\C}{\mathbb{C}}

\newcommand{\COpw}[1]{\textrm{Op}^w_{h, \Phi}\left(#1\right)}

\newcommand{\FBI}{\mathcal{T}_\varphi}

\newcommand{\N}{\mathbb{N}}

\newcommand{\norm}[1]{\|{#1}\|}

\newcommand{\p}{\partial}

\newcommand{\R}{\mathbb{R}}

\newcommand{\set}[2]{\left\{{#1}:{#2}\right\}}

\newcommand{\supp}[1]{\textrm{supp}\left({#1}\right)}

\newcommand{\Opw}[1]{\textrm{Op}^w_h\left({#1}\right)}

\theoremstyle{plain}
\newtheorem{theorem}[subsection]{Theorem}
\newtheorem{proposition}[subsection]{Proposition}
\newtheorem{corollary}[subsection]{Corollary}

\theoremstyle{definition}

\newtheorem{example}{Example}

\newtheorem{remark}{Remark}

\theoremstyle{conjecture}

\theoremstyle{exercise}

\numberwithin{equation}{section}

\title{$L^p$-bounds for eigenfunctions of analytic non-self-adjoint operators with double characteristics}
\author{Francis White}
\affil{University of California Los Angeles}
\date{}

\begin{document}

\maketitle

\begin{abstract}
    We prove sharp uniform $L^p$-bounds for low-lying eigenfunctions of non-self-adjoint semiclassical pseudodifferential operators $P$ on $\R^{n}$ whose principal symbols are doubly-characteristic at the origin of $\R^{2n}$. Our bounds hold under two main assumptions on $P$: (1) the total symbol of $P$ extends holomorphically to a tubular neighborhood of $\R^{2n}$ in $\C^{2n}$, and (2) the quadratic approximation to the principal symbol of $P$ at the origin is elliptic along its singular space. Most notably, our assumptions on the quadratic approximation are less restrictive than those made in prior works, and our main theorem improves the already known results in the case when the symbol of $P$ is analytic.
\end{abstract}


\tableofcontents

\section{Introduction and Statement of Results}

This paper is inspired by the recent progress \cite{Eigenfunction_Bounds} in the understanding of $L^p$-norms of low-lying eigenfunctions of semiclassical pseudodifferential operators with double characteristics. More precisely, we are interested in eigenfunctions of semiclassical pseudodifferential operators on $\R^n$ of the form
\begin{align} \label{operators_we_consider}
	P = \Opw{p_0+hp_1}, \ \ 0 < h \le 1,
\end{align}
where $p_0, p_1 \in C^\infty(\R^{2n})$ belong to a symbol class that we shall specify shortly. Here $0 < h \le 1$ is a semiclassical parameter, and $\Opw{a}$ denotes the semiclassical Weyl quantization of a symbol $a$ on $\R^{2n}$, which is defined formally by
\begin{align}
	\Opw{a} u(x) = \frac{1}{(2\pi h)^n} \int_{\R^n} \int_{\R^n} e^{\frac{i}{h}(x-y) \cdot \xi} a \left(\frac{x+y}{2}, \xi; h \right) u(y) \, dy \, d\xi, \ \ x \in \R^n,
\end{align}
for $u \in \mathcal{S}'(\R^n)$. For background concerning semiclassical Weyl quantization on $\R^n$, we refer to the texts \cite{SemiclassicalAnalysis} and \cite{dimassi_sjostrand}. We make the following assumptions regarding the leading symbol $p_0$:
\begin{enumerate}
	\item $p_0$ is $h$-independent,
	\item $p_0$ has non-negative real part,
	\begin{align} \label{non-negativity of the real part of leading symbol}
		\textrm{Re} \, p_0 \ge 0 \ \textrm{on} \ \R^{2n},
	\end{align}
	with
	\begin{align} \label{null set of the real part is 0}
		(\textrm{Re} \, p_0)^{-1}(0) = \{0\},
	\end{align}
	and
	\item $\textrm{Im} \, p_0$ vanishes to second order at $0 \in \R^{2n}$, i.e.
	\begin{align} \label{imaginary part of principal symbol vanishes to second order}
		\textrm{Im} \, p_0(0) = 0 \ \textrm{and} \ \nabla{(\textrm{Im} \, p_0)}(0) = 0.
	\end{align}
\end{enumerate}
Note that (\ref{non-negativity of the real part of leading symbol}) and (\ref{null set of the real part is 0}) imply that
\begin{align}
	\nabla(\textrm{Re} \, p_0)(0) = 0,
\end{align}
and hence $p_0$ vanishes to second order at $0 \in \R^{2n}$,
\begin{align} \label{doubly characteristic at 0}
	p_0(0) = 0 \ \textrm{and} \ \nabla p_0(0) = 0.
\end{align}
We summarize (\ref{doubly characteristic at 0}) by saying that $p_0$ is \emph{doubly characteristic} at $0 \in \R^{2n}$. To state our assumptions concerning the growth of $p_0$ and $p_1$ at infinity, we first recall that an {\emph{order function}} on $\R^{2n}$ is a Lebesgue measurable function $m: \R^{2n} \rightarrow (0,\infty)$ such that
\begin{align}
	\exists C>0, \ \exists N>0: \ m(X) \le C \langle X - Y \rangle^{N} m(Y), \ \ X, Y \in \R^{2n}.
\end{align}
Here $\langle X \rangle = (1+\abs{X}^2)^{1/2}$ denotes the Japanese bracket of $X \in \R^{2n}$. For any order function $m$ on $\R^{2n}$, we have the symbol class $S(m)$ consisting of all $a: \R^{2n} \times (0,1] \rightarrow \C$ such that
\begin{align}
	a(\cdot; h) \in C^\infty(\R^{2n}), \ \ 0<h \le 1,
\end{align}
and
\begin{align}
\forall \alpha \in \N^{2n}, \, \exists C>0: \ \abs{\p^\alpha_X a(X;h)} \le Cm(X), \ \ X \in \R^{2n}, \ \ 0< h \le 1.
\end{align}
Regarding the symbols $p_0$ and $p_1$, we assume that there is an order function $m$ on $\R^{2n}$ with
\begin{align} \label{conditions on order function}
	 m \in S(m) \ \textrm{and} \ m \ge 1 
\end{align}
such that
\begin{align}
	p_0, p_1 \in S(m).
\end{align}
We also assume that $\textrm{Re} \, p_0$ is elliptic at infinity in the sense that
\begin{align} \label{real part is elliptic at infinity}
	\textrm{Re} \, p_0(X) \ge \frac{1}{C} m(X), \ \ \abs{X} \ge C,
\end{align}
for some $C>0$.

Operators of the form (\ref{operators_we_consider}) with $p_0$ and $p_1$ satisfying the above assumptions include, for example, Schr\"{o}dinger operators on $\R^n$ with complex potentials. Other examples of such semiclassical operators may be found in areas of mathematical physics such as fluid dynamics, superconductivity, and kinetic theory. See, for instance, \cite{SpectralAnalysis}, \cite{Davies}, and \cite{WitLap}. In many applications, the operator $P$ arises as an unbounded operator $L^2(\R^n)$, and one is interested in determining the spectrum of $P$ in the semiclassical limit as $h \rightarrow 0^+$. We may realize $P$ as an unbounded operator on $L^2(\R^n)$ with the domain
\begin{align} \label{the domain}
	H_h(m) := \textrm{Op}^w_h(m)^{-1} \left(L^2(\R^n) \right),
\end{align}
for $h>0$ sufficiently small. When equipped with the domain (\ref{the domain}), the operator $P$ becomes a closed and densely defined operator on $L^2(\R^n)$. Concerning the spectrum of $P$, the assumption (\ref{real part is elliptic at infinity}) that $\textrm{Re} \, p_0$ is elliptic at infinity implies that there is $0<h_0 \le 1$ and $C>0$ such that 
\begin{align} \label{discreteness of low-lying spectrum}
\begin{split}
&\textrm{Spec}(P) \cap D(0,C) \ \textrm{is discrete and consists of eigenvalues} \\ 
&\textrm{of finite algebraic multiplicity for every} \ 0< h \le h_0.
\end{split}
\end{align}
Here $D(0,r)$ denotes the open disc of radius $r$ in $r$ centered at the origin $0$. For a proof of (\ref{discreteness of low-lying spectrum}), see Section 3 of \cite{Hager_Sjostrand}. The eigenvalues $\lambda(h)$ of $P$ that lie in a disc of the form $D(0,Ch)$ for some $C>0$ are known as \emph{low-lying eigenvalues} of $P$. Thanks to the works \cite{ParaMultiChar}, \cite{BoutetHypo}, \cite{Stolk_Herau_Sjostrand}, and \cite{EigenvaluesAndSubelliptic}, complete asymptotic expansions for the low-lying eigenvalues of $P$ are known when Weyl symbol of $P$ admits an asymptotic expansion in the class $S(m)$, i.e. there exists a sequence of $h$-independent symbols $p_{1,j} \in S(m)$, $j \in \N$, such that
\begin{align}
	p_1 \sim \sum_{j=0}^\infty h^j p_{1,j} \ \textrm{in} \ S(m),
\end{align}
and the quadratic approximation to $p_0$ at $0 \in \R^{2n}$,
\begin{align} \label{quadratic_approx_at_0}
	q(X) = \frac{1}{2} p''_0(0) X \cdot X, \ \ X \in \R^{2n},
\end{align}
satisfies suitable partial ellipticity hypotheses. Note that the sign assumption (\ref{non-negativity of the real part of leading symbol}) implies that the complex-valued quadratic form $q$ defined by (\ref{quadratic_approx_at_0}) has non-negative real part,
\begin{align} 
	\textrm{Re} \, q \ge 0.
\end{align}

In this paper, we are primarily interested in the eigenfunctions that correspond to low-lying eigenvalues of $P$. A \emph{low-lying eigenfunction of $P$} is a family $u = u(h) \in L^2(\R^n)$, $0<h \le 1$, such that
\begin{align}
	\begin{cases}
		P u(h) = z(h) u(h), \\
		\norm{u(h)}_{L^2(\R^n)} = 1,
	\end{cases}
\end{align}
where $z(h) \in \C$ satisfies
\begin{align}
	z(h) = \mathcal{O}(h), \ \ h \rightarrow 0^+.
\end{align}
When the operator $P$ is non-self-adjoint, little is known about the low-lying eigenfunctions of $P$. In particular, relationship between the quadratic form $q$ and properties of the low-lying eigenfunctions of $P$ is not well-understood. In this paper, our objective shall be to establish $L^p$-bounds for low-lying eigenfunctions of $P$ under minimal partial ellipticity assumptions on $q$.

\begin{remark}
We remark the study of $L^p$-bounds of eigenfunctions of self-adjoint semiclassical pseudodifferential operators is well-established and has a long history. For more information, we refer to the works \cite{Koch_Tataru_Zworski}, \cite{Koch_Tataru}, and the references therein.
\end{remark}

To the best of our knowledge, the first and only work to undertake a study of $L^p$-bounds for low-lying eigenfunctions of non-self-adjoint semiclassical pseudodifferential operators of the above type has been the work \cite{Eigenfunction_Bounds}, which showed, under the assumption that $\textrm{Re} \, q$ is positive-definite, that any low-lying eigenfunction $u$ of $P$ satisfies the bound
\begin{align} \label{eigenfunction_bound_Krupchyk_Uhlmann}
	\norm{u(h)}_{L^p(\R^n)} \le \mathcal{O}(1)h^{\frac{n}{2p}-\frac{n}{4}}, \, \ \ h \rightarrow 0^+,
\end{align}
for $p$ in the range $2 \le p \le \infty$. In particular, the bound (\ref{eigenfunction_bound_Krupchyk_Uhlmann}) is saturated by the eigenfunctions of the semiclassical harmonic oscillator on $\R^n$, and thus the bounds (\ref{eigenfunction_bound_Krupchyk_Uhlmann}) are sharp within the class of operators considered in \cite{Eigenfunction_Bounds}. 

In the present work, we will show that the bounds (\ref{eigenfunction_bound_Krupchyk_Uhlmann}) also hold in cases where $\textrm{Re} \, q$ need not be positive-definite, provided  we assume in addition that $p_0$ and $p_1$ admit suitable holomorphic extensions to a tubular neighborhood of $\R^{2n}$ in $\C^{2n}$. This kind of result is of interest because in many physical applications the real quadratic form $\textrm{Re} \, q$ does indeed fail to be positive-definite. Such a situation arises in, for example, the study of operators Kramers-Fokker-Planck type in kinetic theory. For more information, see e.g. \cite{Stolk_Herau_Sjostrand}. Actually, our main result will be somewhat stronger than what we have just described. Namely, we will show that the low-lying eigenfunctions of semiclassical operators of the kind we consider satisfy (\ref{eigenfunction_bound_Krupchyk_Uhlmann}) for $p$ in the entire range $1 \le p \le \infty$. Thus our main result improves that of \cite{Eigenfunction_Bounds} in the case when $\textrm{Re} \, q$ is positive-definite and $p_0$ and $p_1$ belong to a suitable holomorphic symbol class.

In the work \cite{EigenvaluesAndSubelliptic}, asymptotic expansions for the low-lying eigenvalues of $P$ were established under the assumption that the quadratic approximation $q$ to $p_0$ at $0 \in \R^{2n}$ is elliptic only along a certain, distinguished, subspace of $\R^{2n}$, known as the \emph{singular space} of $q$. In the present work, we aim to establish $L^p$-bounds for low-lying eigenfunctions of $P$ under this very same partial ellipticity assumption on $q$. In order to state our main result, we pause to recall the notion of the singular space of a complex-valued quadratic form $q$ on $\R^{2n}$ with non-negative real part $\textrm{Re} \, q \ge 0$.

Let $\R^{2n}$ be equipped with the standard symplectic form
\begin{align} \label{Hamilton matrix of q}
    \sigma((x,\xi), (y,\eta)) = \xi \cdot y - x \cdot \eta, \ \ \ (x,\xi), \ (y,\eta) \in \R^{2n}.
\end{align}
Suppose $q: \R^{2n} \rightarrow \C$ is a complex-valued quadratic form with non-negative real part $\textrm{Re} \, q \ge 0$. Let $q(\cdot, \cdot)$ denote the symmetric $\C$-bilinear polarization of $q$. Because $\sigma$ is non-degenerate, there is a unique $F \in \textrm{Mat}_{2n \times 2n}(\C)$ such that
\begin{align*}
    q((x,\xi), (y,\eta)) = \sigma((x,\xi), F(y,\eta)), \ \ (x,\xi), (y,\eta) \in \R^{2n}.
\end{align*}
This matrix $F$ is called the \emph{Hamilton map} or \emph{Hamilton matrix of $q$} (see Section 21.5 of \cite{HormanderIII}). The \emph{singular space $S$ of $q$} is defined as the following finite intersection of kernels:
\begin{align} \label{definition of singular space}
    S = \left(\bigcap_{j=0}^{2n-1} \ker{\left[(\textrm{Re} \ F)(\textrm{Im} \ F)^j\right]} \right) \cap \R^{2n}.
\end{align}
The singular space was first introduced by M. Hitrik and K. Pravda-Starov in \cite{QuadraticOperators} where it arose naturally in the study of spectra and semi-group smoothing properties for non-self adjoint quadratic differential operators. The concept of the singular space has since been shown to play a key role in the understanding of hypoelliptic and spectral properties of non-elliptic quadratic differential operators. See for instance \cite{SemiclassicalHypoelliptic}, \cite{EigenvaluesAndSubelliptic}, \cite{ContractionSemigroup}, \cite{SubellipticEstimatesQuadraticDifferentialOperators}, \cite{NonEllipticQuadraticFormsandSemiclassicalEstimates}, and \cite{SpectralProjectionsAndResolventBounds}. The notion of the singular space is also crucial for the description of the propagation of global microlocal singularities for Schr\"{o}dinger equations on $\R^n$ with quadratic Hamiltonians. The interested reader may consult \cite{GaborSingularities}, \cite{PolynomialSingularities}, \cite{time_dependent}, \cite{ExponentialSingularities}, and \cite{Global_Analytic}, as well as \cite{PartialGS}, \cite{alphonse2020polar}, and \cite{Lp_bounds}.

Next, we recall the definition of the holomorphic symbol class $S_{\textrm{Hol}}(m)$. If $m$ is an order function on $\R^{2n}$ and $W$ is a bounded open neighborhood of $0 \in \C^{2n}$, we may extend $m$ to a function $\tilde{m}$ on the tubular neighborhood $\R^{2n} + W$ of $\R^{2n}$ in $\C^{2n}$ by setting
\begin{align}
	\tilde{m}(Z) = m(\textrm{Re}(Z)), \ \ Z \in \R^{2n} + W.
\end{align}
We define $S_{\textrm{Hol}}(m)$ as the set of all $a: \R^{2n} \times (0,1] \rightarrow \C$ for which there is a bounded open neighborhood $W$ of $0$ in $\C^{2n}$ and a function $\tilde{a}: (\R^{2n}+W) \times (0,1] \rightarrow \C$ extending $a$ such that
\begin{align}
	\tilde{a}(\cdot; h) \in \textrm{Hol}(\R^{2n} + W), \ \ 0< h \le 1,
\end{align}
and
\begin{align}
	\exists C>0: \ \abs{\tilde{a}(Z;h)} \le C \tilde{m}(Z), \ \ Z \in \R^{2n}+W, \ 0<h\le 1.
\end{align}
Note that by Cauchy's inequalities, the derivatives of symbols in $S_{\textrm{Hol}}(m)$ are controlled by $\tilde{m}$ in any strictly smaller tubular neighborhood of $\R^{2n}$ as well: if $a \in S_{\textrm{Hol}}(m)$, then for any open $\tilde{W} \subset \subset W$ we have
\begin{align}
	\forall \alpha \in \N^{2n}, \ \exists C = C(\alpha)>0: \ \abs{\p^\alpha_Z \tilde{a}(Z;h)} \le C \tilde{m}(Z), \ \ Z \in \R^{2n} + \tilde{W}, \ \ 0<h \le 1,
\end{align}
where $\tilde{a}$ is the holomporphic extension of $a$ and $\p_{Z} = \frac{1}{2} \left(\p_{\textrm{Re} \, Z} - i \p_{\textrm{Im} \, Z} \right)$. For additional background concerning the holomorphic symbol class $S_{\textrm{Hol}}(m)$, see Section 12.3 of \cite{Lectures_on_Resonances}.

The following theorem is the main result of this work, which establishes $L^p$ bounds for low-lying eigenfunctions of the operator $P$ defined in (\ref{operators_we_consider}) under the assumption that $p_0, p_1 \in S_{\textrm{Hol}}(m)$ and that the quadratic form $q$ introduced in (\ref{quadratic_approx_at_0}) is elliptic along its singular space $S$, i.e.
\begin{align} \label{elliptic along singular space definition}
	q(X) = 0, \ X \in S \implies X = 0.
\end{align}
In formulating our theorem, we shall be equivalently concerned with $L^2$-normalized solutions $u = u(h)$ of an equation of the form $P u = 0$. In the sequel, we shall refer to a family $u = u(h) \in L^2(\R^n)$ satisfying
\begin{align}
\begin{split}
	\begin{cases}
		Pu = 0 \ \textrm{on} \ \R^n, \\
		\norm{u}_{L^2(\R^n)} = 1,
	\end{cases}
\end{split}
\end{align}
for all $0<h \le 1$, as a \emph{ground state} for the operator $P$.



\begin{theorem} \label{main_theorem}
	Suppose that $P = \textrm{Op}^w_h(p_0 + hp_1)$, where $p_0, p_1 \in S_{\textrm{Hol}}(m)$ for an order function $m$ on $\R^{2n}$ satisfying (\ref{conditions on order function}). Assume that $p_0$ is $h$-independent and satisfies (\ref{non-negativity of the real part of leading symbol}), (\ref{null set of the real part is 0}), (\ref{imaginary part of principal symbol vanishes to second order}), and (\ref{real part is elliptic at infinity}). Assume, in addition, that the quadratic approximation $q$ to $p_0$ at $0\in \R^{2n}$, defined in (\ref{quadratic_approx_at_0}), is elliptic along its singular space $S$ in the sense (\ref{elliptic along singular space definition}). If $u = u(h) \in L^2(\R^n)$ is such that
	\begin{align} \label{ground_state_equation}
	\begin{split}
	\begin{cases}
		Pu=0 \ \ \textrm{on} \ \ \R^n, \ \ n \ge 1, \\
		\norm{u}_{L^2(\R^n)} = 1,
	\end{cases}
	\end{split}
	\end{align}
	for all $0<h \le 1$, then there exists $0<h_0 \le 1$ such that
	\begin{align} \label{main_bound}
		\norm{u}_{L^p(\R^n)} \le \mathcal{O}(1)h^{\frac{n}{2p}-\frac{n}{4}}, \ \ \ 1 \le p \le \infty,
	\end{align}
	for all $0< h \le h_0$.
\end{theorem}



\begin{example} \label{Harmonic Oscillator Example}
	As noted in \cite{Eigenfunction_Bounds}, the bounds (\ref{main_bound}) are saturated by the $L^2$-normalized eigenfunctions of the quantum harmonic oscillator on $\R^n$. Let
	\begin{align}
		P = -h^2 \Delta + \abs{x}^2, \ \ x \in \R^n, \ \ n \ge 1.
	\end{align}
	The operator $P$, when viewed as an unbounded operator on $L^2(\R^n)$ with the domain
	\begin{align}
		\mathcal{D}(P) = \set{u \in L^2(\R^n)}{x^\alpha \p^\beta_x u \in L^2(\R^n), \ \abs{\alpha+\beta} \le 2},
	\end{align}
	is self-adjoint and has a discrete spectrum. Explicitly, the eigenvalues of $P$ are given by
	\begin{align}
		\lambda_\alpha(h) = (2 \abs{\alpha} + n)h, \ \ \alpha \in \N^n,
	\end{align}
	and the corresponding $L^2$-normalized eigenfunctions have the form
	\begin{align}
		u_\alpha(h)(x) = h^{-\frac{n}{4}} p_\alpha(h^{-\frac{1}{2}} x) e^{-\frac{\abs{x}^2}{2h}}, \ \ x \in \R^n,
	\end{align}
	where $p_\alpha$ is a Hermite polynomial of degree $\abs{\alpha}$. For more information, see Section 6.1 of \cite{SemiclassicalAnalysis}. An explicit computation gives that
	\begin{align}
		\norm{u_\alpha(h)}_{L^p(\R^n)} = C(\alpha, p) h^{\frac{n}{2p} - \frac{n}{4}}, \ \ 0 < h \le 1,
	\end{align}
	where
	\begin{align}
		C(\alpha, p) := \norm{p_\alpha(\cdot) e^{-\frac{\abs{\cdot}^2}{2}}}_{L^p(\R^n)}, \ \ \alpha \in \N^n, \ \ 1 \le p \le \infty.
	\end{align}
	Since
	\begin{align}
		P = \textrm{Op}^w_h(p_0),
	\end{align}
	for
	\begin{align}
		p_0(x,\xi) = \abs{x}^2 + \abs{\xi}^2, \ \ (x,\xi) \in \R^{2n},
	\end{align}
	which is a non-negative elliptic quadratic form on $\R^{2n}$, we see that the bound (\ref{main_bound}) is sharp within the class of semiclassical pseudodifferential operators on $\R^n$ satisfying the hypotheses of Theorem \ref{main_theorem}.
\end{example}

\begin{example}
	A more general class of semiclassical operators to which Theorem \ref{main_theorem} applies are Sch\"{o}dinger operators on $\R^n$ with analytic complex-valued potentials. Let
	\begin{align} \label{Schrodinger operator with complex potential}
		P = -h^2 \Delta + V(x) \ \ \textrm{on} \ \ \R^n, \ \ n \ge 1,
	\end{align}
	where $V \in C^\infty(\R^n; \C)$ satisfies the following assumptions:
	\begin{enumerate}
		\item $\textrm{Re} \, V \ge 0$,
		\item $(\textrm{Re} \, V)^{-1}(0) = \{0\}$,
		\item $(\textrm{Im} \, V)(0) = 0$ and $\nabla(\textrm{Im} \, V)(0) = 0$,
		\item $\det{V''(0)} \neq 0$,
		\item there exists $s \ge 0$ such that
			\begin{align} \label{ellipticity of the potential}
				\textrm{Re} \, V(x) \ge \frac{1}{C} \abs{x}^s, \ \ \abs{x} \ge C,
			\end{align}
			for some $C>0$, and
		\item there exists $\epsilon > 0$ and a holomorphic extension $\tilde{V} \in \textrm{Hol}(\R^n + i(-\epsilon, \epsilon)^n)$ of $V$ satisfying
		\begin{align}
			\abs{\tilde{V}(z)} \le C \langle \textrm{Re} \, z \rangle^s, \ \ z \in \R^n + i(-\epsilon, \epsilon)^n,
		\end{align}
		for some $C>0$, where $s$ is as in (\ref{ellipticity of the potential}). 
	\end{enumerate}
	We may view the operator $P$ as a closed, densely defined operator on $L^2(\R^n)$ with the the maximal domain
	\begin{align}
		\mathcal{D}(P) = \set{u \in L^2(\R^n)}{(-h^2 \Delta + V(x)) u \in L^2(\R^n)}.
	\end{align}
	The spectrum of $P$ in an open disc $D(0,Ch)$ of radius $C h$, centered at $0 \in \C$, is discrete, and the low-lying eigenfunctions of $P$ correspond to eigenvalues $z(h) \in D(0,Ch)$.
	
	Let
	\begin{align}
	p_0(x,\xi) = \abs{\xi}^2 + V(x), \ \ (x,\xi) \in \R^{2n}.	
	\end{align}
	We observe that
	\begin{align}
		P = \textrm{Op}^w_h(p_0).
	\end{align}
	Thanks to our assumptions on $V$, the symbol $p_0$ admits the holomorphic extension
	\begin{align}
		\tilde{p}_0(z,\zeta) = \zeta^2 + \tilde{V}(z), \ \ (z,\zeta) \in \R^{2n} + i(-\epsilon, \epsilon)^{2n},
	\end{align}
	for any $\epsilon>0$. Clearly
	\begin{align}
		\textrm{Re} \, p_0 \ge 0 \ \ \textrm{on} \ \ \R^{2n},
	\end{align}
	with
	\begin{align}
		p_0(0) = 0, \ \ \nabla p_0(0) = 0,
	\end{align}
	and we have
	\begin{align}
		\tilde{p}_0 \in S_{\textrm{Hol}}(m)
	\end{align}
	for the order function
	\begin{align}
		m(x,\xi) = \langle{\xi}\rangle^2 + \langle x \rangle^s, \ \ (x,\xi) \in \R^{2n},
	\end{align}
	where $s$ is as in (\ref{ellipticity of the potential}). Also, from (\ref{ellipticity of the potential}), we see that there is $C>0$ such that
	\begin{align}
		\textrm{Re} \, p_0(x,\xi) \ge \frac{1}{C} m(x,\xi), \ \ \abs{(x,\xi)} \ge C.
	\end{align}
	The Hessian of $p_0$ at $0 \in \R^{2n}$ is
	\begin{align}
		p_0''(0) =
		\begin{pmatrix}
			V''(0) & 0 \\
			0 & 2I_n
		\end{pmatrix},
	\end{align}
	where $I_n$ denotes the identity matrix of size $n \times n$. Let
	\begin{align}
		q(Z) = \frac{1}{2} p_0''(0) Z \cdot Z, \ \ Z \in \C^{2n},
	\end{align}
	be the quadratic approximation to $p_0$ at $0 \in \C^{2n}$. A straightforward computation shows that the Hamilton matrix of $q$ is
	\begin{align}
		F =
		\begin{pmatrix}
			0 & I \\
			-\frac{1}{2} V''(0) & 0
		\end{pmatrix}.
	\end{align}
	Thus
	\begin{align}
		\textrm{Re} \, F = 		
		\begin{pmatrix}
			0 & I \\
			-\frac{1}{2} (\textrm{Re} \, V)''(0) & 0
		\end{pmatrix}, \ \
		\textrm{Im} \, F =
		\begin{pmatrix}
			0 & 0 \\
			-\frac{1}{2} (\textrm{Im} \, V)''(0) & 0
		\end{pmatrix},
	\end{align}
	and hence
	\begin{align}
	(\textrm{Re} \, F)(\textrm{Im} \, F) =
	\begin{pmatrix}
		-\frac{1}{2} (\textrm{Im} \, V)''(0) & 0 \\
		0 & 0
	\end{pmatrix}, \ \ 
		(\textrm{Re} \, F)(\textrm{Im} \, F)^j = 0 \ \textrm{for all} \ j \ge 2.
	\end{align}
	It follows that
	\begin{align}
		\ker{\left(\textrm{Re} \, F \right)} = \ker{\left[(\textrm{Re} \, V)''(0) \right]} \times \{0\}, \ \ \ker{\left[(\textrm{Re} \, F)(\textrm{Im} \, F)\right]} = \ker{\left[(\textrm{Im} \, V)''(0)\right]} \times \C^n,
	\end{align}
	and
	\begin{align}
		\ker{\left[(\textrm{Re} \, F)(\textrm{Im} \, F)^j\right]} = \C^{2n} \ \textrm{for all} \ j \ge 2.
	\end{align}
	We thus see that the singular space of $q$ is
	\begin{align}
		S = \left[\ker{\left(V''(0)\right)} \cap \R^n \right] \times \{0\}.
	\end{align}
	Because $\det{V''(0)} \neq 0$, we may deduce that the singular space of $q$ is trivial,
	\begin{align}
		S = \{0\}.
	\end{align}
	In particular, $q$ is elliptic along $S$. By Theorem \ref{main_theorem}, any low-lying eigenfunction $u$ of $P$ must satisfy the bounds (\ref{main_bound}).
	
	We point out that computation of the singular space $S$ above is essentially a special case of Lemma 2.2 in \cite{Hitrik_Bellis}. We refer to this work for the related topic of magnetic Schr\"{o}dinger operators with complex-valued potentials.
	
\end{example}

\begin{remark} \label{comparison remark}
	Theorem \ref{main_theorem} improves the main result of \cite{Eigenfunction_Bounds} in several ways when it is assumed in addition that the symbols $p_0$ and $p_1$ belong to the holomorphic symbol class $S_{\textrm{Hol}}(m)$. Most significantly, the bound (\ref{main_bound}) holds under the weaker assumption that the quadratic approximation $q$ to $p_0$ at $0\in \R^{2n}$ is elliptic only along its singular space $S$. It is also of note that the work \cite{Eigenfunction_Bounds} only establishes the bound (\ref{main_bound}) for $p$ in the range $2 \le p \le \infty$. By contrast, the approach we give in this paper yields the bound (\ref{main_bound}) for $p$ in the entire range $1 \le p \le \infty$. That the bound (\ref{main_bound}) should also hold for $p$ in the range $1 \le p < 2$ is actually a very reasonable expectation in view of Example \ref{Harmonic Oscillator Example}. Finally, in \cite{Eigenfunction_Bounds}, it is assumed that the symbols $p_0$ and $p_1$ grow at most quadratically as $\abs{X} \rightarrow \infty$, i.e.
	\begin{align} \label{first prior assumption}
		\norm{\p^\alpha p_0}_{L^\infty(\R^n)}, \ \norm{\p^\alpha p_1}_{L^\infty(\R^n)} = \mathcal{O}(1), \ \ h \rightarrow 0^+, \ \ \abs{\alpha} \ge 2,
	\end{align}
	and that $\textrm{Re} \, p_0$ grows at least quadratically as $\abs{X} \rightarrow \infty$, i.e. there exists $C,c>0$ such that
	\begin{align} \label{second prior assumption}
		\textrm{Re} \, p_0(X) \ge c \langle X \rangle^2, \ \ \abs{X} \ge C.
	\end{align}
	Our main result, on the other hand, applies to operators whose symbols may grow faster than $\langle X \rangle^2$ at infinity. For example, Theorem \ref{main_theorem} applies to low-lying eigenfunctions of the operator
	\begin{align}
		P = -h^2 \Delta + \abs{x}^2 + h^4 \Delta^2 + \abs{x}^4, \ \ x \in \R^n, \ \ n \ge 1,
	\end{align}
	whose Weyl symbol
	\begin{align}
		p_0(x,\xi) = \abs{x}^2 + \abs{\xi}^2 + \abs{x}^4 + \abs{\xi}^4, \ \ (x,\xi) \in \R^{2n},
	\end{align}
	has size comparable to $\langle (x,\xi) \rangle^4$ when $|(x,\xi)| \rightarrow \infty$.
\end{remark}


We conclude this introduction with an informal overview of the proof of Theorem \ref{main_theorem}. Since the symbol $p_0$ is elliptic away from $0$, the ground states $u$ of $P$ are well microlocalized to any small but fixed neighborhood of $0$ in the phase space $\R^{2n}$. This fact may be expressed conveniently in terms of any global metaplectic Fourier-Bros-Iagolnitzer (FBI) transform of $u$. Let $\varphi$ be a holomorphic quadratic form on $\C^{2n} = \C^n_z \times \C^n_y$ such that
\begin{align}
	\det{\varphi''_{zy}} \neq 0, \ \ \textrm{Im} \, \varphi''_{yy}>0.
\end{align}
We refer to such a holomorphic quadratic form as an \emph{FBI phase function on $\C^{2n}$} (\cite{SemiclassicalAnalysis}, \cite{Minicourse}, \cite{GlobalILagrangians}). Let $\mathcal{T}_\varphi: \mathcal{S}'(\R^n) \rightarrow \textrm{Hol}(\C^n)$, given by
\begin{align} \label{associated_FBI_transform}
	\mathcal{T}_\varphi v(z) = c_\varphi h^{-3n/4} \int_{\R^n} e^{i \varphi(z,y) / h} v(y) \, dy, \ \ v \in \mathcal{S}'(\R^n),
\end{align}
where
\begin{align}
	c_\varphi = 2^{-n/2} \pi^{-3n/4} (\det{\textrm{Im} \ \varphi''_{yy}})^{-1/4} \abs{\det{\varphi''_{zy}}},
\end{align}
be the corresponding metaplectic FBI transform on $\R^n$. Let $\C^{2n} = \C^n_z \times \C^n_\zeta$ be equipped with the standard complex symplectic form $\sigma = d\zeta \wedge dz \in \Lambda^{(2,0)}(\C^{2n})$. We recall that $\mathcal{T}_\varphi$ is a Fourier integral operator associated to the complex linear canonical transformation $\kappa_\varphi: \C^{2n} \rightarrow \C^{2n}$ given implicitly by
\begin{align}
	\kappa_\varphi: (w, -\p_w \varphi(z,w)) \mapsto (z,\p_z \varphi(z,w)), \ \ (z,w) \in \C^{2n},
\end{align}
and that $\mathcal{T}_\varphi$ maps $L^2(\R^n)$ unitarily onto the Bargmann space
\begin{align}
	H_{\Phi_{0}}(\C^n) = L^2(\C^n, e^{-2\Phi_0(z)/h} \, L(dz)) \cap \textrm{Hol}(\C^n).
\end{align}
Here $L(dz)$ is the Lebesgue measure on $\C^n$, and $\Phi_0(z)$ is the strictly plurisubharmonic quadratic form on $\C^n$ given by
\begin{align} \label{definition_of_the_psh_weight}
	\Phi_0(z) = \textrm{max}_{y\in \R^n} \left(-\textrm{Im} \ \varphi(z,y) \right), \ \ z \in \C^n.
\end{align}
We also recall that $\kappa_\varphi$ maps the real phase space $\R^{2n}$ isomorphically onto the $I$-Lagrangian, $R$-symplectic subspace
\begin{align}
	\Lambda_{\Phi_0} := \set{\left(z, \frac{2}{i} \p_z \Phi_0(z)\right)}{z\in \C^n},
\end{align}
of $\C^{2n}$. Let $\pi_1 : \C^{2n} \rightarrow \C^n$ be the projection onto the first factor in $\C^{2n}$, $\pi_1: (z,\zeta) \mapsto z$, and observe that $\pi_1$ restricts to an $\R$-linear isomorphism $\Lambda_{\Phi} \rightarrow \C^n$. Thus the map
\begin{align}
	\kappa_{\varphi}^\flat := \pi_1 \circ \left. \kappa_{\varphi}\right|_{\R^{2n}}
\end{align}
 is an $\R$-linear isomorphism $\R^{2n} \rightarrow \C^n$. Since $u$ is microlocalized near $0$ in $\R^{2n}$, it follows from standard arguments that the $L^1$-mass of the function $\mathcal{T}_\varphi u(z) e^{-\Phi_0(z)/h}$ with respect to $L(dz)$ is $\mathcal{O}(h^\infty)$ in the complement of any small but fixed neighborhood of $0$ in $\C^n$ (see Proposition \ref{localization_lemma_FBI_side} below). By combining this observation with the identity $u = \mathcal{T}_\varphi^* \mathcal{T}_\varphi u$ and applying the Minkowski integral inequality, one can show that for any $N \in \N$, $\delta>0$, and $1 \le p \le \infty$, there is $C = C(\delta, N, p)>0$ and $0<h_0 \le 1$, independent of $p$, such that
\begin{align} \label{L^p_bound_from_FBI_side}
	\norm{u}_{L^p(\R^n)} \le C h^{\frac{n}{2p}-\frac{3n}{4}} \int_{\abs{z} < \delta} \abs{\mathcal{T}_\varphi u(z)} e^{-\Phi_0(z)/h} \, L(dz) + Ch^N, \ \ 0< h \le h_0.
\end{align}
For a proof, see Proposition \ref{first_bound_Lp_norm} below.

Establishing the validity of (\ref{L^p_bound_from_FBI_side}) for any FBI phase function $\varphi$ is the first main step in the proof of Theorem \ref{main_theorem}. The next step is to show that there exists an FBI phase function $\varphi$ on $\C^{2n}$ and a real analytic strictly plurisubharmonic function $\Phi^* \in C^{\omega}(\textrm{neigh}(0; \C^n); \R)$ such that
\begin{align} \label{quadratic_decay_near_0}
	\Phi_0(z) - \Phi^*(z) \ge c \abs{z}^2, \ \ z \in \textrm{neigh}(0; \C^n),
\end{align}
for some $c>0$, and
\begin{align} \label{microlocal regularity}
	\norm{\mathcal{T}_\varphi u}^2_{L^2_{\Phi^*}(\textrm{neigh}(0; \C^n))} := \int_{\textrm{neigh}(0; \C^n)} \abs{\mathcal{T}_\varphi u(z)}^2 e^{-2 \Phi^*(z) / h} \, L(dz) = \mathcal{O}(1), \ \ h \rightarrow 0^+.
\end{align}
If such a weight $\Phi^*$ can be found, then, for $\delta>0$ small enough, we have
\begin{align} \label{Gaussian_near_0}
\begin{split}
\int_{\abs{z} < \delta} \abs{\mathcal{T}_\varphi u(z)} e^{-\Phi_0(z) / h} \, L(dz)	 &= \int_{\abs{z} < \delta} \abs{\mathcal{T}_\varphi u(z)} e^{-\Phi^*(z)/h} e^{-c \abs{z}^2 / h} \, L(dz)
\le \mathcal{O}(1) h^{\frac{n}{2}},
\end{split}
\end{align}
by the Cauchy-Schwart inequality. Combining (\ref{L^p_bound_from_FBI_side}) and (\ref{Gaussian_near_0}) gives the desired bound (\ref{main_bound}).

The construction of the strictly plurisubharmonic weight $\Phi^*$ defined near $0$ in $\C^{n}$ satisfying (\ref{quadratic_decay_near_0}) and (\ref{microlocal regularity}) is purely dynamical. We give a rough sketch of the procedure. Let
\begin{align}
	\mathfrak{p}_0 := p_0 \circ \kappa_{\varphi}^{-1} \in \textrm{Hol}(\Lambda_{\Phi_0}+W),
\end{align}
where $W \subset \subset \C^{2n}$ is a sufficiently small open neighborhood of $0$ in $\C^{2n}$. Let
\begin{align*}
H_{\mathfrak{p}_0} = \mathfrak{p}'_{0, \zeta} \cdot \p_z - \mathfrak{p}'_{0, z} \cdot \p_{\zeta} \in T^{(1,0)}(\Lambda_{\Phi_0}+W)
\end{align*}
be the complex Hamilton vector field of $\mathfrak{p}_0$. For $t \in \C$ with $\abs{t}$ sufficiently small, we may define the complex-time Hamilton flow $\kappa_t = \exp{(tH_{\mathfrak{p}_0})}$ of $\mathfrak{p}_0$ in a neighborhood of $0 \in \C^{2n}$ as follows:
\begin{align}
(z(t), \zeta(t)) = \kappa_t(z_0, \zeta_0), \ \ (z_0, \zeta_0) \in \textrm{neigh}(0; \C^{2n}), \ \ t \in \textrm{neigh}(0; \C),
\end{align}
if and only if $z(t)$ and $\zeta(t)$ satisfy the complex Hamilton's equations
\begin{align}
	\begin{split}
	\begin{cases}
		\p_t z(t) = \p_\zeta \mathfrak{p}_0(z(t), \zeta(t)), \\
		\p_t \zeta(t) = -\p_z \mathfrak{p}_0(z(t), \zeta(t)), \\
		z(0) = z_0, \ \zeta(0) = \zeta_0,
	\end{cases}
	\end{split}
\end{align}
where $\p_t = \frac{1}{2}(\p_{\textrm{Re} \, t} - i \p_{\textrm{Im} \, t})$. From the work \cite{WF_Multiple}, it is known that there is $0< T \ll 1$ such that
\begin{align}
	\kappa_t (\Lambda_{\Phi_0} \cap \, \textrm{neigh}(0; \C^{2n})) \cap \textrm{neigh}(0; \C^{2n}) = \Lambda_{\Phi_t}, \ \ 0 \le \abs{t} < T,
\end{align}
where $(\Phi_t)_{\abs{t} < T}$ is a family of real-valued, strictly plurisubharmonic functions defined in a neighborhood of $0 \in \C^n$ solving the complex-time eikonal equation
\begin{align} \label{complex_eikonal_equation}
\begin{split}
\begin{cases}
	2 \p_t \Phi_t(z) + i \mathfrak{p}_0 \left(z, \frac{2}{i} \p_z \Phi_t(z) \right) = 0, \ \ \abs{t}<T, \ \ z \in \textrm{neigh}(0; \C^n), \\
	\left. \Phi_t \right|_{t=0}(z) = \Phi_0(z), \ \ z \in \textrm{neigh}(0; \C^n),
\end{cases}
\end{split}
\end{align}
and $(\Lambda_{\Phi_t})_{\abs{t} < T}$ is the family of $I$-Lagrangian, $R$-symplectic submanifolds of $\C^{2n}$ given by
\begin{align}
	\Lambda_{\Phi_t} := \set{\left(z, \frac{2}{i} \p_z \Phi_t(z) \right)}{z \in \textrm{neigh}(0; \C^n)}, \ \ \abs{t}<T.
\end{align}
For a proof, see the discussion in Section 3 below. In particular, since $\mathfrak{p}_0$ vanishes to second order at $0 \in \C^{2n}$ and $\Phi_0(0) = 0$, and we have that
\begin{align*}
	\Phi_0 - \Phi_t \ \textrm{vanishes to $2^{\textrm{nd}}$ order at $z=0$ for all $0 \le \abs{t} < T$}.
\end{align*}
For $\abs{t} < T$, let $\Xi_t$ be the quadratic approximation to $\Phi_t$ at $0$, i.e. $\Xi_t$ is the real quadratic form on $\C^n$ that begins the Taylor expansion of $\Phi_t$ at $0 \in \C^n$,
\begin{align}
	\Phi_t(z) = \Xi_t(z) + \mathcal{O}(\abs{z}^3), \ \ \abs{z} \rightarrow 0^+, \ \ 0 \le \abs{t} < T.
\end{align}
Taylor expanding (\ref{complex_eikonal_equation}) to second order about $z=0$ shows that $\Xi_t$ must satisfy the quadratic complex-time eikonal equation
\begin{align} \label{quadratic_eikonal_equation}
\begin{split}
\begin{cases}
	2 \p_t \Xi_t(z) + i \mathfrak{q} \left(z, \frac{2}{i} \p_z \Xi_t(z) \right) = 0, \ \ \abs{t}<T, \ \ z \in \C^n, \\
	\left. \Xi_t \right|_{t=0} = \Phi_0 \ \textrm{on} \ \C^n,
\end{cases}
\end{split}	
\end{align}
where $\mathfrak{q} := q \circ \kappa_{\varphi}^{-1}$ is the quadratic approximation to $\mathfrak{p}_0$ at $0$, i.e. 
\begin{align}
	\mathfrak{p}_0(Z) = \mathfrak{q}(Z) + \mathcal{O}(\abs{Z}^3), \ \ Z \in \C^{2n}, \ \ \abs{Z} \rightarrow 0^+.
\end{align}
Using the assumption that $q$ is elliptic along $S$, one can show that there exists an FBI phase function $\varphi$ on $\C^{2n}$ with the property that for any $0<T_0 < T$ there exists a non-zero complex time $t_0$ with $\abs{t_0} < T_0$ such that
\begin{align} \label{good time quadratic estimate}
	\Phi_0(z) - \Xi_{t_0}(z) > 0, \ \ z \in \C^n \backslash \{0\}.
\end{align}
A proof of this claim is given in Proposition \ref{good complex time proposition} below. Taking
\begin{align} \label{definition_of_tilde_phi}
	\Phi^* := \Phi_{t_0},
\end{align}
we obtain a strictly plurisubharmonic weight $\Phi^*$ defined in a neighborhood of $0$ in $\C^n$ so that (\ref{quadratic_decay_near_0}) holds. Now, by Egorov's theorem, $\mathcal{T}_\varphi u$ satisfies the FBI-side pseudodifferential equation
\begin{align} \label{FBI side equation}
	\textrm{Op}^w_{\Phi_0, h}(\mathfrak{p}_0 + h \mathfrak{p}_1) \mathcal{T}_\varphi u = 0, \ \ 0 < h \le 1,
\end{align}
where $\mathfrak{p}_j := p_j \circ \kappa_{\varphi}^{-1} \in \textrm{Hol}(\Lambda_{\Phi_0}+W)$, $j=0,1$, and $\textrm{Op}^w_{\Phi_0, h}(\mathfrak{p}_0 + h \mathfrak{p}_1)$ denotes the complex Weyl quantization of $\mathfrak{p}_0+h\mathfrak{p}_1$ with respect to the weight $\Phi_0$ (see the discussion in Section 2 below). From (\ref{FBI side equation}) and (\ref{complex_eikonal_equation}), one may deduce that $\mathcal{T}_\varphi u$ satisfies the following dynamical bound in a sufficiently small neighborhood of $0$ in $\C^n$:
\begin{align} \label{locally in evolved spaces intro}
	\exists 0<T_0 < T, \ \exists \delta>0, \ \exists 0 < h_0 \le 1: \  \ \ \sup_{\substack{0 \le \abs{t} < T_0 \\ 0< h \le h_0}} \norm{\mathcal{T}_\varphi u}_{L^2_{\Phi_t}( \{\abs{z} < \delta \})} < \infty.
\end{align}
The proof of (\ref{locally in evolved spaces intro}), whose main ingredient is the quantization-multiplication theorem for pseudodifferential operators with holomorphic symbols (\cite{Minicourse}, \cite{Lectures_on_Resonances}, \cite{Sjostrand_Geometric}), is given in Section 3 below. In particular, from (\ref{locally in evolved spaces intro}), it follows that
\begin{align}
	\sup_{0< h \le h_0} \norm{\mathcal{T}_\varphi u}_{L^2_{\Phi_{t_0}}(\{\abs{z} < \delta \})} < \infty,
\end{align}
where $0< \abs{t_0} < T_0$ is such that (\ref{good time quadratic estimate}) holds. Thus the weight $\Phi^*$ defined by (\ref{definition_of_tilde_phi}) satisfies (\ref{quadratic_decay_near_0}) and (\ref{microlocal regularity}).


The plan for this paper is as follows. In Section 2, we prove that the ground states $u$ of $P$ are well microlocalized to any fixed neighborhood of $0$ in $\R^{2n}$ and we establish bounds for FBI transforms of $u$. In Section 3, we give a self-contained derivation of the complex-time eikonal equation (\ref{complex_eikonal_equation}), and we prove the dynamical bound (\ref{locally in evolved spaces intro}). In Section 4, we conclude the proof of Theorem \ref{main_theorem} as outlined in this introduction by proving the existence of an FBI phase function $\varphi$ on $\C^{2n}$ and a small non-zero complex time $t_0$ such that the weight $\Phi^*$ defined by (\ref{definition_of_tilde_phi}) satisfies (\ref{quadratic_decay_near_0}).

\vspace{5mm}

{\bf{Acknowledgements.}} The author would like to thank Michael Hitrik for reading an early draft of this manuscript and offering helpful feedback and suggestions.


\section{Microlocalization of the Ground States}


In this section, we establish that any ground state $u = u(h) \in L^2(\R^n)$ of the operator $P$ is well-microlocalized to any small, but fixed, neighborhood of $0 \in \R^{2n}$. We begin with an elementary parametrix construction, which we carry out using the calculus of semiclassical pseudodifferential operators on $\R^n$ with $C^\infty$-symbols. For background on semiclassical pseudodifferential calculus on $\R^n$, including standard notation, we refer to Chapters 4 and 8 of \cite{SemiclassicalAnalysis}.

\begin{proposition} \label{course_microlocalization_real_side}
	Let $p_0, p_1 \in S(m)$, where $m$ is an order function on $\R^{2n}$ satisfying (\ref{conditions on order function}). Assume that $p_0$ is $h$-independent and that $p_0$ satisfies (\ref{non-negativity of the real part of leading symbol}), (\ref{null set of the real part is 0}), (\ref{imaginary part of principal symbol vanishes to second order}), and (\ref{real part is elliptic at infinity}). If $P = \textrm{Op}^w_h(p_0+hp_1)$ and $u = u(h) \in L^2(\R^n)$ is such that
	\begin{align}
		\begin{split}
		\begin{cases}
			P u = 0 \ \textrm{on} \ \R^n, \ n \ge 1, \\
			\norm{u}_{L^2(\R^n)} = 1,
		\end{cases}	
		\end{split}
	\end{align}
	for all $0<h\le 1$, then, for any $\delta>0$, there exists $\psi \in C^\infty_0(\R^{2n})$ with $\textrm{supp} \, \psi \subset \set{X \in \R^{2n}}{\abs{X} < \delta}$ and there exists $R = \mathcal{O}_{\mathcal{S}' \rightarrow \mathcal{S}}(h^\infty)$ such that
	\begin{align} \label{precise_statement_localization}
		u = \Opw{\psi}u + Ru, \ \ 0<h\le h_0,
	\end{align}
	for some $0< h_0 \le 1$.
\end{proposition}
\begin{proof}
Let $\delta>0$ be arbitrary and let $\chi \in C^\infty(\R^{2n}; [0,1])$ be such that 
\begin{align}
	\chi(X) \equiv 1, \ \ \abs{X} \le \frac{\delta}{2},
\end{align}
and
\begin{align}
	\textrm{supp} \, \chi \subset \set{X \in \R^{2n}}{\abs{X} < \delta}.
\end{align}
Set
\begin{align}
	\tilde{p}:= p_0 + hp_1 + \chi, \ \ 0< h \le 1.
\end{align}
Note that $\tilde{p} \in S(m)$. Our objective is to construct a parametrix for the operator $\Opw{\tilde{p}}$. The assumption (\ref{real part is elliptic at infinity}) that $\textrm{Re} \, p_0$ is elliptic in the class $S(m)$ implies that there is $c>0$ such that
\begin{align}
	\textrm{Re} \, \tilde{p}(X;h) \ge c m(X), \ \ X \in \R^{2n},
\end{align}
for all $h>0$ sufficiently small, depending on $\delta$. Thus the symbol
\begin{align}
e_0 := \frac{1}{\tilde{p}}
\end{align}
is well-defined and belongs to the class $S(\frac{1}{m})$. By the semiclassical composition calculus, there is a symbol $r_0 \in S(1)$ such that
\begin{align}
	\Opw{e_0} \Opw{\tilde{p}} = I+h \Opw{r_0},
\end{align}
As a consequence of the semiclassical Calderon-Vaillancourt theorem (see Theorem 4.23 in \cite{SemiclassicalAnalysis}), we have $\Opw{r_0} = \mathcal{O}_{L^2 \rightarrow L^2}(1)$ as $h \rightarrow 0^+$. Thus there is $0<h_0 \le 1$ such that the operator $I+h \Opw{r_0}$ is boundedly invertible on $L^2(\R^n)$ for all $0<h \le h_0$. It follows that
\begin{align} \label{inverse_to_other_side}
	(I+h\Opw{r_0})^{-1} \Opw{e_0} \Opw{\tilde{p}} = I, \ \ 0<h \le h_0.
\end{align}
By Beals' Theorem (see the discussion on page 177 of \cite{SemiclassicalAnalysis}), there is $\tilde{r}_0 \in S(1)$ such that
\begin{align} \label{Beals_lemma}
	\Opw{\tilde{r}_0} = (I+h\Opw{r_0})^{-1}
\end{align}
for all $h>0$ sufficiently small. By the composition calculus, we have
\begin{align} \label{composition_calculus_again}
	\Opw{\tilde{r}_0} \Opw{e_0} = \Opw{e_1},
\end{align}
where $e_1 = \tilde{r}_0 \# e_0 \in S(\frac{1}{m})$ is the Moyal product of $\tilde{r}_0$ and $e_0$. From (\ref{inverse_to_other_side}), (\ref{Beals_lemma}), and (\ref{composition_calculus_again}), we deduce that
\begin{align}
	\Opw{e_1} \Opw{\tilde{p}} = I.
\end{align}
Since
\begin{align}
	\Opw{\tilde{p}} = P + \Opw{\chi},
\end{align}
we have
\begin{align} \label{parametrix_applied_to_ground_state}
	u = \Opw{e_1} \Opw{\tilde{p}}u = \Opw{e_1} \Opw{\chi} u
\end{align}
for all $h>0$ sufficiently small. Let $\psi \in C^\infty_0(\R^{2n})$ be such that 
\begin{align}
	\psi \equiv 1 \ \textrm{in} \ \textrm{neigh}(\textrm{supp} \, \chi; \R^{2n}), \ \ \textrm{supp} \, \psi \subset \set{X \in \R^{2n}}{\abs{X}<\delta}.
\end{align}
Using (\ref{parametrix_applied_to_ground_state}), we get
\begin{align} \label{reproduction_of_u}
\begin{split}
	u &= \Opw{\psi} \Opw{e_1} \Opw{\chi} u + \Opw{1-\psi} \Opw{e_1} \Opw{\chi} u \\
	&= \Opw{\psi} u + \Opw{1-\psi} \Opw{e_1} \Opw{\chi} u.
\end{split}
\end{align}
Because $\textrm{supp} \, \chi$ is compact and $\textrm{supp} \, (1-\psi) \cap \textrm{supp} \, \chi = \emptyset$, it follows from the general theory that there is $r \in h^\infty \mathcal{S}(\R^n)$ such that
\begin{align} \label{definition_of_remainder}
	R:=\Opw{(1-\psi)} \Opw{e_1} \Opw{\chi} = \Opw{r}.
\end{align}
Hence
\begin{align} \label{R is regularizing}
	R = \mathcal{O}_{\mathcal{S}' \rightarrow \mathcal{S}}(h^\infty).
\end{align}
From (\ref{reproduction_of_u}) and (\ref{definition_of_remainder}), we conclude that
\begin{align}
u = \textrm{Op}^w_h(\psi) u + R u, \ \ 0 < h \le h_0.	
\end{align}
The proof of Proposition \ref{course_microlocalization_real_side} is complete. 
\end{proof}


We next explore the consequences of Proposition \ref{course_microlocalization_real_side} for any FBI transform of a ground state $u$ of $P$. Let $\C^{2n} = \C^n_\zeta \times \C^n_z$ be equipped with the standard complex symplectic form $\sigma = d\zeta \wedge dz \in \Lambda^{(2,0)}(\C^{2n})$. Let $\varphi = \varphi(z,y)$ be a holomorphic quadratic form on $\C^{2n} = \C^n_z \times \C^n_y$ such that
\begin{align} \label{FBI phase function again}
\det{\varphi''_{zy}} \neq 0, \ \ \textrm{Im} \, \varphi''_{yy}>0,	
\end{align}
and let $\mathcal{T}_\varphi: \mathcal{S}'(\R^n) \rightarrow \textrm{Hol}(\C^n)$ be its associated FBI transform introduced in (\ref{associated_FBI_transform}). Let
\begin{align} \label{associated PSH weight}
\Phi_0(z) := \textrm{max}_{y \in \R^n} (-\textrm{Im} \, \varphi(z,y)), \ \ z \in \C^n.	
\end{align}
As the maximum of a family of pluriharmonic quadratic forms on $\C^n$, the function $\Phi_0$ is a plurisubharmonic quadratic form on $\C^n$. In fact, the quadratic form $\Phi_0$ is strictly plurisubharmonic, i.e. $\Phi''_{0,\overline{z} z}>0$. See \cite{SemiclassicalAnalysis} or \cite{Minicourse} for a proof. In the sequel, we shall refer to the quadratic form $\Phi_0$ defined by (\ref{associated PSH weight}) as the {\emph{strictly plurisubharmonic weight associated to the FBI phase function $\varphi$}}. Let
\begin{align} \label{definition_of_I_Lagrangian}
\Lambda_{\Phi_0} := \set{\left(z, \frac{2}{i} \p_z \Phi_0(z) \right) \in \C^{2n}}{z \in \C^n},	
\end{align}
and let
\begin{align} \label{Bargmann space}
	H_{\Phi_0}(\C^n) = L^2(\C^n, e^{-2 \Phi_0(z)/h} \, L(dz)) \cap \textrm{Hol}(\C^n)
\end{align}
be the Bargmann space of entire functions on $\C^n$ associated to the weight $\Phi_0$. Here $L(dz)$ denotes the Lebesgue measure on $\C^n$. It is well known that $\mathcal{T}_\varphi$ is unitary $L^2(\R^n) \rightarrow H_{\Phi_0}(\C^n)$. For a proof, see Theorem 13.7 in \cite{SemiclassicalAnalysis}, Theorem 1.3.3 in \cite{Minicourse}, or Section 12.2 of \cite{Lectures_on_Resonances}.

Following the discussion in Section 12.2 of \cite{Lectures_on_Resonances}, we recall that an order function on $\Lambda_{\Phi_0}$ is a Lebesgue measurable function $\mathfrak{m}: \Lambda_{\Phi_0} \rightarrow (0,\infty)$ such that
\begin{align}
	\exists C>0, \ \exists N \in \R: \ \mathfrak{m}(Z) \le C \langle Z - W \rangle^N \mathfrak{m}(W), \ \ Z,W \in \Lambda_{\Phi_0}.
\end{align}
Given an order function $\mathfrak{m}$ on $\Lambda_{\Phi_0}$, we may introduce the symbol class $S(\Lambda_{\Phi_0}, \mathfrak{m})$ consisting of all $a \in C^\infty(\Lambda_{\Phi_0})$ such that
\begin{align}
\forall \alpha, \beta \in \N^n, \ \exists C>0: \ \abs{\p^\alpha_z \p^\beta_{\overline{z}} \left(a\left(z, \frac{2}{i} \p_z \Phi_0(z) \right) \right)} \le C \mathfrak{m} \left(z, \frac{2}{i} \p_z \Phi_0(z)\right), \ \ z \in \C^n.
\end{align}
Also, if $\mathfrak{m}$ is an order function on $\Lambda_{\Phi_0}$, then we may define the Sobolev space
\begin{align}
H_{\Phi_0, \mathfrak{m}}(\C^n) = L^2\left(\C^n, \mathfrak{m} \left(z, \frac{2}{i} \p_z \Phi_0(z)\right)^2 e^{-2 \Phi_0(z)/h} \, L(dz)\right) \cap \textrm{Hol}(\C^n),
\end{align}
which is a Hilbert space equipped with the norm
\begin{align}
	\norm{v}^2_{L^2_{\Phi_0, \mathfrak{m}}(\C^n)} = \int_{\C^n} \abs{v(z)}^2 \mathfrak{m}\left(z, \frac{2}{i} \p_z \Phi_0(z) \right)^2 e^{-2 \Phi_0(z)/h} \, L(dz).
\end{align}
Note that $H_{\Phi_0, 1}(\C^n) = H_{\Phi_0}(\C^n)$. Given a symbol $a \in S(\Lambda_{\Phi_0}, \mathfrak{m})$, where $\mathfrak{m}$ is an order function on $\Lambda_{\Phi_0}$, we define the \emph{complex Weyl quantization of $a$} formally by
\begin{align} \label{complex Weyl quantization}
	\textrm{Op}^w_{\Phi_0, h}(a) v(z) = \frac{1}{(2\pi h)^n} \iint_{\Gamma_{\Phi_0}(z)} e^{\frac{i}{h} (z-w) \cdot \zeta} a \left(\frac{z+w}{2}, \zeta \right) v(w) \, dw \wedge d\zeta,
\end{align}
for the contour of integration
\begin{align}
	\Gamma_{\Phi_0}(z): \C^n \ni w \mapsto \zeta = \frac{2}{i} \p_z \Phi_0 \left(\frac{z+w}{2} \right), \ \ z \in \C^n.
\end{align}
Initially, $\textrm{Op}^w_{\Phi_0, h}(a) v$ is defined for $v \in \textrm{Hol}(\C^n)$ such that
\begin{align}
	\forall h \in (0,1], \ \forall N>0, \ \exists C = C(N,h)>0: \ \ \abs{v(z)} \le C \langle z \rangle^{-N} e^{\Phi_0(z) / h}, \ \ z \in \C^n.
\end{align}
One may then show that for any order function $\tilde{\mathfrak{m}}$ on $\Lambda_{\Phi_0}$ the operator $\textrm{Op}^w_{\Phi_0, h}(a)$ defined by (\ref{complex Weyl quantization}) extends by density to an operator $H_{\Phi_0, \tilde{\mathfrak{m}}}(\C^n) \rightarrow H_{\Phi_0, \frac{\tilde{\mathfrak{m}}}{\mathfrak{m}}}(\C^n)$ whose norm is uniformly bounded with respect to $h$, i.e.
\begin{align}
	\textrm{Op}^w_{\Phi_0, h}(a) = \mathcal{O}(1): H_{\Phi_0, \tilde{\mathfrak{m}}}(\C^n) \rightarrow H_{\Phi_0, \frac{\tilde{\mathfrak{m}}}{\mathfrak{m}}}(\C^n), \ \ h \rightarrow 0^+.
\end{align}
For a proof, see Proposition 12.6 in \cite{Lectures_on_Resonances}.

We next review the relationship between the complex Weyl quantization and the ordinary Weyl quantization on $\R^n$. Let $\varphi$ be an FBI phase function on $\C^{2n}$ and let $\Phi_0$ be the strictly plurisubharmonic quadratic form on $\C^n$ associated to $\varphi$. The FBI phase function $\varphi$ generates a complex linear canonical transformation $\kappa_\varphi: \C^{2n} \rightarrow \C^{2n}$ implicitly by
\begin{align} \label{associated canonical transformation}
	\kappa_\varphi: (y, -\varphi'_y(z,y)) \mapsto (z, \varphi'_z(z,y)), \ \ (z,y) \in \C^{2n}.
\end{align}
In the sequel, we refer to $\kappa_\varphi$ as the \emph{complex linear canonical transformation generated by $\varphi$} or as the \emph{complex linear canonical transformation associated to $\varphi$}. From either Theorem 13.5 of \cite{SemiclassicalAnalysis} or Proposition 1.3.2 of \cite{Minicourse}, we know that
\begin{align}
	\kappa_{\varphi}(\R^{2n}) = \Lambda_{\Phi_0},
\end{align}
and we have a version of Egorov's theorem. Namely, if $m$ is any order function on $\R^{2n}$ and $a \in S(m)$, then
\begin{align}
	\mathfrak{a} := a \circ \kappa_{\varphi}^{-1} \in S(\Lambda_{\Phi_0}, \mathfrak{m}),
\end{align}
where $\mathfrak{m}$ is the order function on $\Lambda_{\Phi_0}$ given by
\begin{align}
	\mathfrak{m} = m \circ \kappa_{\varphi}^{-1},
\end{align}
and we have
\begin{align}
	\textrm{Op}^w_h(a) = \mathcal{T}_\varphi^* \circ \textrm{Op}^w_{\Phi_0, h}(\mathfrak{a}) \circ \mathcal{T}_\varphi,
\end{align}
where both sides are viewed as operators on $\mathcal{S}'(\R^n)$. For a proof, see Section 12.2 of \cite{Lectures_on_Resonances}.

Let $P$ and $u$ be as in the statement of Proposition \ref{course_microlocalization_real_side}, and let $\varphi$ be an FBI phase function on $\C^{2n}$ with associated FBI transform $\mathcal{T}_\varphi$ and strictly plurisubharmonic weight $\Phi_0$. In view of Proposition \ref{course_microlocalization_real_side}, it is natural to expect that the mass of the entire function $\mathcal{T}_\varphi u$ will be concentrated near $0 \in \C^n$. Equivalently, we expect that the mass of $\mathcal{T}_\varphi u$ in the complement of any fixed neighborhood of $0$ in $\C^n$ will be semiclassically negligible, i.e. $\mathcal{O}(h^\infty)$. The following proposition solidifies this intuition.


\begin{proposition} \label{localization_lemma_FBI_side}
Let $P$ and $u \in L^2(\R^n)$ be as in the statement of Proposition \ref{course_microlocalization_real_side}, and let $\varphi$ be any FBI phase function on $\C^{2n}$ with associated FBI transform $\mathcal{T}_\varphi$ and associated strictly plurisubharmonic weight $\Phi_0$. For any $\delta>0$, there is $0<h_0 \le 1$ such that for any $1 \le p \le \infty$ we have
\begin{align} \label{mass_away_from_0_is_negligible}
	\norm{\FBI u(z) e^{-\Phi_0(z)/h}}_{L^p(\{\abs{z} \ge \delta\}, L(dz))} \le \mathcal{O}_N(1) h^N, \ \ 0< h \le h_0,
\end{align}
for any $N \in \N$.
\end{proposition}
\begin{proof}
By H\"{o}lder's inequality, it suffices to prove the proposition in the cases $p=1$ and $p=\infty$. We begin by proving (\ref{mass_away_from_0_is_negligible}) in the case when $p=1$. Let $\delta>0$ be given, let $0<\rho<\delta/2$, let $\Lambda_{\Phi_0}$ be as in (\ref{definition_of_I_Lagrangian}), let $\kappa_\varphi: \C^{2n} \rightarrow \C^{2n}$ be the complex linear canonical transformation generated by $\varphi$, let $\pi_1: \C^{2n} \rightarrow \C^n$ be projection onto the first factor $\pi_1: (z,\zeta) \mapsto z$, and let $\kappa^\flat_\varphi := \pi_1 \circ \left. \kappa_{\varphi} \right|_{\R^{2n}}: \R^{2n} \rightarrow \C^n$. By Proposition \ref{course_microlocalization_real_side}, there is $\psi \in C^\infty_0(\R^{2n})$ such that 
\begin{align} \label{size_of_support_of_cutoff}
	\textrm{supp}\left[\psi \circ \left(\kappa^\flat_\varphi\right)^{-1}\right] \subset \set{z \in \C^n}{\abs{z} < \rho},
\end{align}
and $R = \mathcal{O}_{\mathcal{S}' \rightarrow \mathcal{S}}(h^\infty)$ such that
\begin{align} \label{mention previous line}
	u = \textrm{Op}^w_h(\psi) u + R u, \ \ 0 < h \le h_0,
\end{align}
for some $0< h_0 \le 1$. By Egorov's theorem and (\ref{mention previous line}),
\begin{align}
\mathcal{T}_\varphi u = \textrm{Op}^w_{\Phi_0, h} (\tilde{\psi}) \mathcal{T}_\varphi u + \mathcal{T}_\varphi (R u),	
\end{align}
where
\begin{align}
\tilde{\psi} := \psi \circ \kappa_{\varphi}^{-1} \in C^\infty_0(\Lambda_{\Phi_0}).	
\end{align}
Thus
\begin{align} \label{where_we_introduce_I_and_II}
	\int_{\abs{z} \ge \delta} \abs{\FBI u(z)} e^{-\Phi_0(z)/h} \, L(dz) \le {\rm I}(h) + {\rm II}(h),
\end{align}
where
\begin{align} \label{first_definition_of_I}
	{\rm I}(h):=\int_{\abs{z} \ge \delta} \abs{\textrm{Op}^w_{\Phi_0, h}(\tilde{\psi}) \FBI u(z)} e^{-\Phi_0(z)/h} \, L(dz)
\end{align}
and
\begin{align} \label{first_definition_of_II}
	{\rm II}(h) := \int_{\abs{z} \ge \delta} \abs{\mathcal{T}_\varphi (R u)(z)} e^{-\Phi_0(z)/h} \, L(dz). 
\end{align}
We first consider $II(h)$. For $J \in \N$, let $\norm{\cdot}_{J}$ denote the Schwartz seminorm
\begin{align}
	\norm{f}_J := \sum_{\abs{\alpha+\beta} \le J} \norm{x^\alpha \p^\beta f}_{L^\infty(\R^n)}, \ \ f \in \mathcal{S}(\R^n).
\end{align}
From Theorem 13.4 in \cite{SemiclassicalAnalysis}, we know that for every $N>0$ there exists $J \in \N$ and $C>0$ such that
\begin{align}
	\abs{\mathcal{T}_\varphi Ru(z)} \le C \norm{Ru}_{J} h^{-n/4} \langle z \rangle^{-N} e^{\Phi_0(z)/h}, \ \ z \in \C^n, \ \ 0< h \le 1.
\end{align}
Since $R = \mathcal{O}_{\mathcal{S}' \rightarrow \mathcal{S}}(h^\infty)$ and $\norm{u}_{L^2(\R^n)} = 1$, we conclude that for every $M, N>0$, there is $C = C(M,N)>0$ such that
\begin{align} \label{good estimate}
	\abs{\mathcal{T}_\varphi R u(z)} \le C h^M \langle z \rangle^{-N} e^{\Phi_0(z)/h}, \ \ z \in \C^n, \ \  0<h \le h_0.
\end{align}
In particular, for any $M>0$ there is $C>0$ such that
\begin{align} \label{estimate specific choice}
	\abs{\mathcal{T}_\varphi Ru(z)} \le C h^M \langle z \rangle^{-2n-1} e^{\Phi_0(z)/h}, \ \ z \in \C^n, \ \ 0<h \le h_0.
\end{align}
Bounding (\ref{first_definition_of_II}) using (\ref{estimate specific choice}) gives that
\begin{align}
	{\rm II}(h) = \mathcal{O}(h^M)
\end{align}
for any $M>0$. Therefore
\begin{align} \label{II(h) is negligible}
	{\rm II}(h) = \mathcal{O}(h^\infty).
\end{align}

It remains to show that ${\rm I}(h) = \mathcal{O}(h^\infty)$. From the definition (\ref{complex Weyl quantization}) of the complex Weyl quantization, we have
\begin{align} \label{first_integral_to_consider}
	\textrm{Op}^w_{\Phi_0, h}(\tilde{\psi}) \FBI u(z) = \frac{1}{(2\pi h)^n} \iint_{\Gamma_{\Phi_0}(z)} e^{\frac{i}{h}(z-w) \cdot \zeta} \tilde{\psi} \left(\frac{z+w}{2}, \zeta \right) \mathcal{T}_\varphi u(w) \, dw \wedge d\zeta,
\end{align}
where the contour of integration $\Gamma_{\Phi_0}(z)$ is given by
\begin{align} \label{def_Gamma_phi}
	\Gamma_{\Phi_0}(z): \C^n \ni w \mapsto \zeta = \frac{2}{i} \p_z \Phi_0 \left(\frac{z+w}{2} \right), \ \ z \in \C^n.
\end{align}
Let
\begin{align}
	L = \frac{1+(\p^2_{z\overline{z}} \Phi_0)^{-1}\overline{(z-w)} \cdot \p_{\overline{w}}}{1+\abs{z-w}^2/h}, \ \ z, w \in \C^n.
\end{align}
Observe that
\begin{align}
	L e^{\frac{i}{h} (z-w) \cdot \zeta} = e^{\frac{i}{h}(z-w)}, \ \ z, w \in \C^n, \ \ \zeta = \frac{2}{i}\p_z \Phi_0 \left(\frac{z+w}{2} \right).
\end{align}
Thus for any $z \in \C^n$ such that $\abs{z} \ge \delta$ and any $N \in \N$ we have
\begin{align} \label{repeated_integration_by_parts_first_time}
	\textrm{Op}^w_{\Phi_0, h}(\tilde{\psi}) \FBI u(z) = \frac{1}{(2\pi h)^n} \iint_{\Gamma_{\Phi_0}(z)} e^{\frac{i}{h}(z-w) \cdot \zeta} (L^T)^N \tilde{\psi} \left(\frac{z+w}{2}, \zeta \right) \mathcal{T}_\varphi u(w) \, dw \wedge d\zeta,
\end{align}
where $L^T$ denotes the transpose of the differential operator $L$. Since $\tilde{\psi} \in C^\infty_0(\Lambda_{\Phi_0})$,
\begin{align} \label{estimate for transpose}
	(L^T)^N \tilde{\psi} \left(\frac{z+w}{2} \right) = \mathcal{O} \left(\langle h^{-1/2}(z-w)\rangle^{-N}\right), \ \ z,w \in \C^n, \ \ 0<h \le 1,
\end{align}
for any $N \in N$. Now, it is true that
\begin{align} \label{identity_for_real_part_of_phase}
	\textrm{Re} \, \left[i(z-w)\cdot \zeta\right] = \Phi_0(z)-\Phi_0(w), \ \ z, w \in \C^n, \ \zeta = \frac{2}{i} \p_z \Phi_0 \left(\frac{z+w}{2}\right).
\end{align}
For a proof, see Lemma 13.1 of \cite{SemiclassicalAnalysis}. From (\ref{first_definition_of_I}), (\ref{repeated_integration_by_parts_first_time}), (\ref{estimate for transpose}), and (\ref{identity_for_real_part_of_phase}), it follows that for any $N \in \N$ there is $C>0$ such that
\begin{align} \label{first_bound_for_I(h)}
	{\rm I}(h) \le Ch^{-n}\int_{\abs{z} \ge \delta} \int_{\abs{\frac{z+w}{2}} \le \rho} \langle h^{-1/2}(z-w) \rangle^{-N} \abs{\FBI u(w)} e^{-\Phi_0(w)/h} \, L(dw) \, L(dz).
\end{align}
for all $0< h \le h_0$. For any $z, w \in \C^n$ with $\abs{z} \ge \delta$, we have
\begin{align}
\begin{split}
	\abs{\frac{z+w}{2}} \le \rho \implies \abs{z} - \frac{\abs{z-w}}{2} \le \rho \implies 2 \abs{z} - \delta \le \abs{z-w} \implies \abs{z} \le \abs{z-w}.
\end{split}
\end{align}
Thus, from (\ref{first_bound_for_I(h)}) we may deduce that for any $N \in \N$ there is $C>0$ such that
\begin{align} \label{almost finished bound}
	{\rm I}(h) \le Ch^{-n+N/2} \int_{\abs{z} \ge \delta} \int_{\abs{\frac{z+w}{2}} \le \rho} \abs{z}^{-N} \abs{\FBI u(w)} e^{-\Phi_0(w)/h} \, L(dw) \, L(dz)
\end{align}
for all $0< h \le h_0$. Making the change of variables
\begin{align}
	\tilde{z} &= z, \ \ \tilde{w} = \frac{z+w}{2}
\end{align}
in (\ref{almost finished bound}) and applying Fubini's theorem and the Cauchy-Schwarz inequality yields
\begin{align}
	\begin{split}
	{\rm I}(h) &\le C h^{-n+N/2} \int_{\abs{\tilde{z}} \ge \delta} \int_{\abs{\tilde{w}} \le \rho} \abs{\tilde{z}}^{-N} \abs{\mathcal{T}_\varphi u(2\tilde{w}-\tilde{z})} e^{-\frac{\Phi_0(2\tilde{w}-\tilde{z})}{h}} \, L(d\tilde{w}) \, L(d\tilde{z}) \\
	&\le C \norm{u}_{H_{\Phi_0}(\C^n)}h^{-n+N/2} \int_{\abs{\tilde{z}} \ge \delta} \abs{\tilde{z}}^{-N} \, L(d\tilde{z}) \\
	&\le C h^{-n + N/2}, \ \ 0< h \le h_0, 
	\end{split}
\end{align}
for any $N > 2n$. Thus
\begin{align} \label{I(h)_is_negligible}
	{\rm I}(h) = \mathcal{O}(h^\infty).
\end{align}
From (\ref{where_we_introduce_I_and_II}), (\ref{I(h)_is_negligible}), and (\ref{II(h) is negligible}), we conclude
\begin{align} \label{first_estimate_in_proof}
	\int_{\abs{z} \ge \delta} \abs{\mathcal{T}_\varphi u(z)} e^{-\Phi_0(z)/h} \, L(dz) = \mathcal{O}(h^\infty).
\end{align}
Therefore the proposition is true when $p=1$.

Next, we show that (\ref{mass_away_from_0_is_negligible}) holds when $p=\infty$. We recall that the orthogonal projection $\Pi_{\Phi_0}: L^2_{\Phi_0}(\C^n) \rightarrow H_{\Phi_0}(\C^n)$ is given by
\begin{align} \label{definition_Bergman_projector}
	\Pi_{\Phi_0} v(z) = C_{\Phi_0} h^{-n} \int_{\C^n} e^{\frac{2}{h} \Psi_0(z,\overline{w})} v(w) e^{-\frac{2}{h} \Phi_0(w)} \, L(dw), \ \ v \in L^2_{\Phi_0}(\C^n),
\end{align}
where
\begin{align}
	C_{\Phi_0} = \left(\frac{2}{\pi} \right)^n \det{\p^2_{z \overline{z}} \Phi_0},
\end{align}
and $\Psi_0(\cdot, \cdot)$ is the polarization of $\Phi_0$, i.e. $\Psi_0(\cdot, \cdot)$ is the unique holomorphic quadratic form on $\C^{2n}$ such that
\begin{align}
	\Psi_0(z,\overline{z}) = \Phi_0(z), \ \ z \in \C^n,
\end{align}
In the literature, $\Pi_{\Phi_0}$ is known as the \emph{Bergman projector} associated to the strictly plurisubharmonic weight $\Phi_0$. For a proof that the Bergman projector associated to $\Phi_0$ has the integral representation (\ref{definition_Bergman_projector}), we refer the reader to either Theorem 13.6 of \cite{SemiclassicalAnalysis}, Proposition 1.3.4 of \cite{Minicourse}, or Section 12.2 of \cite{Lectures_on_Resonances}. In particular, the real part of the polarization $\Psi_0(\cdot, \cdot)$ satisfies the `fundamental estimate'
\begin{align} \label{fundamental_estimate}
	2 \textrm{Re} \, \Psi_0(z,\overline{w}) - \Phi_0(z) - \Phi_0(w) \asymp -\abs{z-w}^2.
\end{align}
From (\ref{fundamental_estimate}) and the identity $\mathcal{T}_\varphi u = \Pi_{\Phi_0} \mathcal{T}_\varphi u$, it follows that
\begin{align} \label{L^infty_estimate}
	\abs{\FBI u(z) e^{-\Phi_0(z)/h}} \le C_{\Phi_0} h^{-n} \int_{\C^n} e^{-c \abs{z-w}^2/h} \abs{\FBI u(w) e^{-\Phi_0(w)/h}} \, L(dw), \ \ z \in \C^n, \ \ 0< h \le 1.
\end{align}
Let $\delta>0$ be arbitary. For $\abs{z} \ge \delta$, we have
\begin{align} \label{first pointwise bound}
\begin{split}
	\abs{\FBI u(z) e^{-\Phi_0(z)/h}} &\le C_{\Phi_0} h^{-n} \int_{\abs{w} < \delta / 2} e^{-c \delta^2 / 4h} \abs{\FBI u(w) e^{-\Phi_0(w)/h}} \, L(dw) \\
	& \ \ \ + \mathcal{O}(1) h^{-n} \norm{\FBI u(z) e^{-\Phi_0(z)/h}}_{L^1(\{\abs{z} \ge \delta/2\}, L(dz))}.
\end{split}
\end{align}
By the Cauchy-Schwarz inequality,
\begin{align} \label{helper1}
	C_{\Phi_0} h^{-n} \int_{\abs{w} < \delta / 2} e^{-c \delta^2 / 4h} \abs{\FBI u(w) e^{-\Phi_0(w)/h}} \, L(dw) \le \mathcal{O}(1) h^{-n} e^{-c /h} \le \mathcal{O}(1)h^N, \ \ 0< h \le 1,
\end{align}
for any $N \in \N$. Also, since the proposition has already been proven in the case $p=1$, we know that there is $0<h_0 \le 1$ such that
\begin{align} \label{helper2}
	\norm{\FBI u(z) e^{-\Phi_0(z)/h}}_{L^1(\{\abs{z} \ge \delta/2\}, L(dz))} \le \mathcal{O}(1) h^{N}, \ \ 0< h \le h_0,
\end{align}
for any $N \in \N$. From (\ref{first pointwise bound}), (\ref{helper1}), and (\ref{helper2}), we conclude that there is $0<h_0 \le 1$ such that
\begin{align}
	\norm{\FBI u(z) e^{-\Phi_0(z)/h}}_{L^\infty(\{\abs{z} \ge \delta \}, L(dz))} \le \mathcal{O}(1) h^{N}, \ \ 0< h \le h_0,
\end{align}
for all $N \in \N$. Thus (\ref{mass_away_from_0_is_negligible}) holds when $p=\infty$.
\end{proof}

Using Proposition \ref{localization_lemma_FBI_side} and the unitarity of $\mathcal{T}_\varphi$, we can establish a simple upper bound for $\norm{u}_{L^p(\R^n)}$ in terms of the $L^1$-norm of $\mathcal{T}_\varphi u(z) e^{-\Phi_0(z)/h}$ with respect to the Lebesgue measure $L(dz)$ over any bounded neighborhood of $0$ in $\C^n$. 


\begin{proposition} \label{first_bound_Lp_norm}
	Let $u \in L^2(\R^n)$ be as in the statement of Proposition \ref{course_microlocalization_real_side}. Let $\varphi$ be any FBI phase function on $\C^{2n}$ with associated FBI transform $\mathcal{T}_\varphi$ and associated strictly plurisubharmonic weight $\Phi_0$. For any $\delta>0$ and $N \in \N$, there is $0< h_0 \le 1$ such that for any $1 \le p \le \infty$ we have
	\begin{align}
		\norm{u}_{L^p(\R^n)} \le C h^{\frac{n}{2p}-\frac{3n}{4}} \int_{\abs{z} < \delta} \abs{\mathcal{T}_\varphi u(z)} e^{-\Phi_0(z)/h} \, L(dz) + Ch^N, \ \ 0<h\le h_0,
	\end{align}
	for some constant $C = C(\delta, N, p)>0$.
\end{proposition}
\begin{proof}
Let $\varphi$ be an FBI phase function on $\C^{2n}$ with associated FBI transform $\mathcal{T}_\varphi$ and strictly plurisubharmonic weight $\Phi_0$. Let $H_{\Phi_0}(\C^n)$ be the Bargmann space introduced in (\ref{Bargmann space}). Let $0< h_0 \le 1$ be small enough so that the conclusions of Proposition \ref{course_microlocalization_real_side} and Proposition \ref{localization_lemma_FBI_side} hold. Since $\mathcal{T}_\varphi: L^2(\R^n) \rightarrow H_{\Phi_0}(\C^n)$ is unitary, we have
\begin{align} \label{u_rewritten_adjoint_FBI}
	u(x) = \mathcal{T}_\varphi^* \mathcal{T}_\varphi u(x) = c_\varphi^2 h^{-3n/4} \int_{\C^n} e^{-\frac{i}{h} \overline{\varphi(z,x)}} \mathcal{T}_\varphi u(z) e^{-\frac{2}{h} \Phi_0(z)} \, L(dz), \ \ x \in \R^n.
\end{align}
Because $\varphi$ is quadratic and $\textrm{Im} \, \varphi''_{yy} > 0$, for every $z \in \C^n$ there is a unique $x(z) \in \R^n$ such that
\begin{align} \label{maximum achieved}
	\Phi_0(z) = - \textrm{Im} \, \varphi(z, x(z)).
\end{align}
Taking the absolute value of (\ref{u_rewritten_adjoint_FBI}) and using (\ref{FBI phase function again}), (\ref{associated PSH weight}) and (\ref{maximum achieved}), we find that there are constants $C,c>0$ and $0< h_0 \le 1$ such that
\begin{align} \label{absolute_value_u}
	\abs{u(x)} \le C h^{-3n/4} \int_{\C^n} e^{-\frac{c}{h} \abs{x-x(z)}^2} \abs{\mathcal{T}_\varphi u(z)} e^{-\Phi_0(z)/h} \, L(dz), \ \ x \in \R^n,
\end{align}
for all $0< h \le h_0$. Note that the righthand side of (\ref{absolute_value_u}) is finite for every $x \in \R^n$ since Proposition \ref{course_microlocalization_real_side} implies that $u \in \mathcal{S}(\R^n)$ for every $0< h \le h_0$ and hence $\mathcal{T}_\varphi u(z) e^{-\Phi_0(z)/h} = \mathcal{O}(h^{-n/4} \langle z \rangle^{-N})$ for any $N>0$ by Theorem 13.4 of \cite{SemiclassicalAnalysis}. By direct calculation, we have
\begin{align} \label{L^p estimate gaussian}
	\norm{e^{-\frac{c}{h} \abs{\cdot}^2}}_{L^p(\R^n)} = \mathcal{O}(h^{\frac{n}{2p}}),
\end{align}
with the convention that $h^{\frac{n}{2p}} = 1$ when $p=\infty$. Taking the $L^p$-norm on both sides of (\ref{absolute_value_u}), applying Minkowski's integral inequality, and using (\ref{L^p estimate gaussian}) gives
\begin{align}
	\norm{u}_{L^p(\R^n)} \le C h^{\frac{n}{2p}-\frac{3n}{4}} \int_{\C^n} \abs{\mathcal{T}_\varphi u(z)} e^{-\Phi_0(z)/h} \, L(dz), \ \ 0< h \le h_0,
\end{align}
for some $C>0$. Let $\delta>0$ be arbitrary. By Proposition \ref{localization_lemma_FBI_side}, for every $N>0$, there is $C>0$ and $0< h_0 \le 1$ such that
\begin{align}
	\int_{\C^n} \abs{\mathcal{T}_\varphi u(z)} e^{-\Phi_0(z)/h} \, L(dz) = \int_{\abs{z} < \delta} \abs{\mathcal{T}_\varphi u(z)} e^{-\Phi_0(z)/h} \, L(dz) + C h^N, \ \ 0< h \le h_0.
\end{align}
Thus, for every $N \in \N$ and $1 \le p \le \infty$, there is $C = C(\delta, N, p)>0$ such that
\begin{align} \label{estimate_prior_to_metaplectic_operator}
	\norm{u}_{L^p(\R^n)} \le C h^{\frac{n}{2p}-\frac{3n}{4}} \int_{\abs{z} < \delta} \abs{\mathcal{T}_\varphi u(z)} e^{-\Phi_0(z)/h} \, L(dz) + C h^N, \ \ 0 < h \le h_0.
\end{align}
\end{proof}

In the sequel, it will also be useful to have pointwise estimates available for $\mathcal{T}_\varphi u(z) e^{-\Phi_0(z)/h}$ in compact subsets of $\C^n$ that do not contain the origin $0$. When we assume that $p_0, p_1 \in S_{\textrm{Hol}}(m)$, it is actually true that $\mathcal{T}_\varphi u(z) e^{-\Phi_0(z)/h}$ is exponentially small in any fixed compact subset $K \subset \C^n$ that does not contain $0 \in \C^n$. This follows from standard analytic ellipticity arguments, which we review below.

We recall the notion of the semiclassical analytic wavefront set. Let $\varphi$ be an FBI phase function on $\C^{2n}$ with associated FBI transform $\mathcal{T}_\varphi$ and strictly plurisubharmonic weight $\Phi_0$. Let $\kappa_\varphi: \C^{2n} \rightarrow \C^{2n}$ be the complex linear canonical transformation generated by $\varphi$, let $\pi_1: \C^{2n} \rightarrow \C^n$ be projection onto the first factor $\pi_1: (z,\zeta) \mapsto z$, and let $\kappa_\varphi^\flat := \pi_1 \circ \left. \kappa_\varphi \right|_{\R^{2n}}$. For an $h$-dependent family $v = v(h) \in L^2(\R^n)$, $0< h \le 1$, such that $\norm{v}_{L^2(\R^n)} = \mathcal{O}(1)$ as $h \rightarrow 0^+$, we can define the {\emph{semiclassical analytic wavefront set}} $\textrm{WF}_{A,h}(v) \subset \R^{2n}$ of $v$ as follows: a point $(x_0, \xi_0) \in \R^{2n}$ does not lie in $\textrm{WF}_{A,h}(v)$ if there exist $C,c>0$ and a bounded open neighborhood $U_0$ of $\kappa_\varphi^\flat(x_0,\xi_0)$ in $\C^n$ such that
\begin{align}
	\abs{\mathcal{T}_\varphi u(z)} \le C e^{-c/ h + \Phi_0(z)/h}, \ \ z \in U_0, \ \ 0<h \le 1.
\end{align}
It can be shown $\textrm{WF}_{A,h}(v)$ is independent of the choice of FBI phase function $\varphi$. For more information regarding semiclassical analytic wavefront sets, we refer the reader to \cite{AnalyticMicrolocal_Analysis}, \cite{Minicourse}, or \cite{Martinez}. The following corollary characterizes the semiclassical analytic wavefront set of any $L^2$-normalized ground state $u$ of $P$.


\begin{corollary} \label{Wavefront_is_0}
	If $u \in L^2(\R^n)$ is as in the statement of Theorem \ref{main_theorem}, then 
	\begin{align}
		\textrm{WF}_{A,h}(u) = \{0\}.
	\end{align}
Consequently, if $\varphi$ is any FBI phase function on $\C^{2n}$ with associated strictly plurisubharmonic weight $\Phi_0$ and $K$ is any compact subset of $\C^n$ such that $0 \notin K$, then there are $C,c>0$ and $0<h_0 \le 1$ such that
\begin{align} \label{estimate_in_compact_subsets}
	\abs{\mathcal{T}_\varphi u(z)} \le C e^{-c / h + \Phi_0(z)/h}, \ \ z \in K, \ \ 0<h \le h_0.
\end{align}
\end{corollary}
\begin{proof}
Let $P = \textrm{Op}^w_h(p_0 + h p_1)$ be as in the statement of Theorem \ref{main_theorem}. Since $p_0, p_1 \in S_{\textrm{Hol}}(m)$ and $p_0^{-1}(0) = \{0\}$, semiclassical analytic elliptic regularity (see, for example, Theorem 4.2.2 in \cite{Martinez}) implies that
\begin{align}
	\textrm{either} \ \ WF_{A,h}(u) = \emptyset \ \ \textrm{or} \ \ WF_{A,h}(u) = \{0\}.
\end{align}
Let $\varphi$ be an FBI phase function on $\C^{2n}$ with associated FBI transform $\mathcal{T}_\varphi$ and strictly plurisubharmonic weight $\Phi_0$. Suppose towards contradiction that $\textrm{WF}_{A,h}(u) = \emptyset$. Then there exist $C,c,\delta>0$ and $0<h_0 \le 1$ such that
\begin{align} \label{small near 0}
	\sup_{\abs{z} < \delta} \abs{\mathcal{T}_\varphi u(z) e^{-\Phi_0(z)/h}} \le C e^{-c/h}, \ \ 0 < h \le h_0.
\end{align}
Proposition \ref{first_bound_Lp_norm} with $p=2$ gives
\begin{align} \label{small away from 0}
	\int_{\abs{z} \ge \delta} \abs{\mathcal{T}_\varphi u(z)}^2 e^{-2 \Phi_0(z)/h} \, L(dz) = \mathcal{O}(h^\infty).
\end{align}
Combining (\ref{small near 0}) and (\ref{small away from 0}) and using the unitarity of $\mathcal{T}_\varphi$ gives that
\begin{align}
	\norm{u}_{L^2(\R^n)} = \mathcal{O}(h^\infty),
\end{align}
which is impossible since $\norm{u}_{L^2(\R^n)} = 1$ for all $0<h \le 1$. Therefore, it must be the case that
\begin{align}
	\textrm{WF}_{A,h}(u) = \{0\}.
\end{align}
If $K \subset \C^n$ is a compact subset such that $0 \notin K$, then we can find a finite collection of bounded open subsets $U_j$, $1 \le j \le k$, of $\C^n$, constants $C_j, c_j>0$, $1 \le j \le k$, and $0< h_0 \le 1$ such that
\begin{enumerate}
	\item $0 \notin U_j$, $1 \le j \le k$,
	\item $K \subset \bigcup_{1 \le j \le k} U_j$, and
	\item $\abs{\mathcal{T}_\varphi u(z)} \le C_j e^{ {-c_j / h} + \Phi_0(z)/h}, \ \ z \in U_j, \ \ 1 \le j \le k, \ \ 0<h \le h_0$.
\end{enumerate}
Let $U = \bigcup_{1 \le j \le k} U_j$. We have
\begin{align}
	\abs{\mathcal{T}_\varphi u(z)} \le C e^{-c/ h + \Phi_0(z)/h}, \ \ z \in U, \ \ 0<h \le h_0,
\end{align}
where $C = \textrm{max}_{1 \le j \le k} C_j > 0$ and $c = \textrm{min}_{1 \le j \le k} c_j >0$. Therefore $\mathcal{T}_\varphi u$ satisfies (\ref{estimate_in_compact_subsets}).
\end{proof}



\section{Dynamical Bounds on the FBI Transform Side}

We begin this section a short review of Hamiltonian dynamics in the complex domain. For a textbook treatment of these concepts, see Chapter 11 of \cite{AnalyticMicrolocal_Analysis}. Here $\C^{2n} = \C^n_z \times \C^n_\zeta$ is equipped with the standard holomorphic symplectic form 
\begin{align} \label{definition of the standard form}
	\sigma = d\zeta \wedge dz \in \Lambda^{(2,0)}(\C^{2n}).
\end{align}
Let $X \subset \C^{2n}$ be open and let $f \in \textrm{Hol}(X)$. The \emph{complex Hamilton vector field} of $f$ is defined as the unique holomorphic vector field $H_f \in T^{1,0} (X)$ so that
\begin{align} \label{abstract_definition_complex_Hamilton_field}
	H_f \lrcorner \, \sigma = -df \ \textrm{in} \ X.
\end{align}
Explicitly,
\begin{align} \label{concrete_definition_complex_Hamilton_field}
	H_f = \p_\zeta f \cdot \p_z - \p_z f \cdot \p_\zeta \ \textrm{in} \ X.
\end{align}
From (\ref{abstract_definition_complex_Hamilton_field}), it is clear that
\begin{align} \label{interior_multiplication_complex_Hamilton_vector_field}
	\sigma(t, H_f) = df(t), \ \ t \in T^{1,0}(X).
\end{align}
Since $df(t) = 0$ whenever $t \in T^{0,1}(X)$, the identity (\ref{interior_multiplication_complex_Hamilton_vector_field}) also holds for every $t \in T(X) \otimes \C$. In particular, (\ref{interior_multiplication_complex_Hamilton_vector_field}) holds whenever $t \in T(X)$. To $H_f$ we associate the real vector field
\begin{align} \label{definition_corresponding_real_field}
	\widehat{H_f} = H_f + \overline{H_f} \in T(X),
\end{align}
which is the unique real vector field on $X$ such that $H_f - \widehat{H_f} \in T^{0,1}(X)$. Because $\sigma$ is a holomorphic $(2,0)$-form,
\begin{align}
	\sigma(t, \widehat{H_f}-H_f) = 0, \ \ t \in T(X) \otimes \C,
\end{align}
and hence
\begin{align} \label{associated_real_field_has_the_correct_property}
	\widehat{H_f} \lrcorner \, \sigma = -df \ \textrm{in} \ X.
\end{align}
When we separate the real and imaginary parts of (\ref{associated_real_field_has_the_correct_property}), we obtain
\begin{align} \label{associated_real_field_properties}
	\widehat{H_f} \lrcorner \, \textrm{Re} \, \sigma = -d(\textrm{Re} \, f), \ \ \ \widehat{H_f} \lrcorner \, \textrm{Im} \, \sigma = -d(\textrm{Im} \, f).
\end{align}
Letting $H^{\textrm{Re} \, \sigma}_{\textrm{Re} \, f}$ and $H^{\textrm{Im} \, \sigma}_{\textrm{Im} \, f}$ denote Hamilton vector fields of $\textrm{Re} \, f$ and $\textrm{Im} \, f$ with respect to the real symplectic forms $\textrm{Re} \, \sigma$ and $\textrm{Im} \, \sigma$ on $\C^{2n}$, respectively, we conclude from (\ref{associated_real_field_properties}) that
\begin{align} \label{key identity for complex time flow}
	\widehat{H_f} = H^{\textrm{Re} \, \sigma}_{\textrm{Re} \, f} = H^{\textrm{Im} \, \sigma}_{\textrm{Im} \, f}.
\end{align}
Repeating this discussion with $f$ replaced by $if$ gives
\begin{align} \label{Hamilton_field_identity_I_need}
	\widehat{H_{if}} = H^{\textrm{Im} \, \sigma}_{\textrm{Re} \, f} = -H^{\textrm{Re} \, \sigma}_{\textrm{Im} \, f},
\end{align}
where $H^{\textrm{Im} \, \sigma}_{\textrm{Re} \, f}$, resp. $H^{\textrm{Re} \, \sigma}_{\textrm{Im} \, f}$, is the Hamilton vector field of $\textrm{Re} \, f$, resp. $\textrm{Im} \, f$, with respect to $\textrm{Im} \, \sigma$, resp. $\textrm{Re} \, \sigma$.

Following \cite{AppFIO}, \cite{CplxPhaseFtn}, and \cite{AnalyticMicrolocal_Analysis}, we define the \emph{complex-time Hamilton flow} $\kappa_t$, $t \in \C$, of $f$ as the complex local $1$-parameter family of holomorphic diffeomorphisms of $X$ given by
\begin{align} \label{formal definition of complex-time Hamilton flow}
	\kappa_t = \exp{(1\widehat{H_{tf}})}, \ \ t \in \C.
\end{align}
Here $\exp{(1 \widehat{H_{tf}})}$ denotes the flow of the real vector field $\widehat{H_{tf}}$ in $X$ at time $1$. By abuse of notation, we shall often write $\exp{(tH_f)}$ in place of $\exp{(1\widehat{H_{tf}})}$ so that the complex-time Hamilton flow $\kappa_t$ generated by $H_f$ may be expressed as
\begin{align}
	\kappa_t = \exp{(tH_{f})}, \ \ t \in \C.
\end{align}
Thanks to the Cauchy-Riemann equations, we have
\begin{align} \label{commuting vector fields}
	[\widehat{H_f}, \widehat{H_{if}}] = 0
\end{align}
for any $f \in \textrm{Hol}(X)$. Also, if $f \in \textrm{Hol}(X)$, it is true that
\begin{align}
	\widehat{H_{tf}} = (\textrm{Re} \, t) \widehat{H_f} + (\textrm{Im} \, t) \widehat{H_{if}}
\end{align}
for any $t \in \C$. This observation, in conjunction with (\ref{formal definition of complex-time Hamilton flow}) and (\ref{commuting vector fields}), shows that the complex-time Hamilton flow $\kappa_t$ on $X$ generated by $H_f$ is given by
\begin{align}
	\kappa_t = \exp{\left[(\textrm{Re} \, t) \widehat{H_f}\right]} \exp{\left[(\textrm{Im} \, t) \widehat{H_{if}} \right]}, \ \ t \in \C,
\end{align}
whenever the righthand side is defined. It follows that the flow $\kappa_t$ is given explicitly by
\begin{align}
	Z(t) = \kappa_t(Z_0), \ \ Z_0 \in X, \ \ t \in \textrm{neigh}(0; \C),
\end{align}
if and only if $Z(t)$ satisfies
\begin{align}
\begin{split}
	\begin{cases}
	\p_{\textrm{Re} \, t} Z(t) = \widehat{H_f}(Z(t)) = (\p_\zeta f(Z(t)), -\p_z f(Z(t))), \\
		\p_{\textrm{Im} \, t} Z(t) = \widehat{H_{if}}(Z(t)) = (i\p_\zeta f(Z(t)), -i\p_z f(Z(t))), \\
		Z(0) = Z_0.
	\end{cases}
\end{split}
\end{align}
A straightforward verification shows that this description of the flow $\kappa_t$ is equivalent to 
\begin{align}
(z(t), \zeta(t)) = \kappa_t(z_0, \zeta_0), \ \ (z_0, \zeta_0) \in X, \ \ t \in \textrm{neigh}(0; \C),
\end{align}
if and only if $z(t)$ and $\zeta(t)$ satisfy the complex Hamilton's equations
\begin{align} \label{complex_Hamilton_equations}
	\begin{split}
	\begin{cases}
		\p_t z(t) = \p_\zeta f(z(t), \zeta(t)), \\
		\p_t \zeta(t) = -\p_z f(z(t), \zeta(t)), \\
		z(0) = z_0, \ \zeta(0) = \zeta_0,
	\end{cases}
	\end{split}
\end{align}
where $\p_t = \frac{1}{2}(\p_{\textrm{Re} \, t} - i \p_{\textrm{Im} \, t})$. From (\ref{complex_Hamilton_equations}), it is easy to check that $\kappa_t$ preserves the complex symplectic form $\sigma$,
\begin{align}
	\kappa_t^* \sigma = \sigma, \ \ t \in \C,
\end{align}
whenever the lefthand side is well-defined. For this reason, we say that $\kappa_t$ is a \emph{complex local $1$-parameter family of complex canonical transformations} of $X$.

Let $p_0 \in S_{\textrm{Hol}}(m)$ be as in the statement of Theorem \ref{main_theorem}, let $\varphi$ be an FBI phase function on $\C^{2n}$ with associated FBI transform $\mathcal{T}_\varphi$, let $\Phi_0$ be the strictly plurisubharmonic weight associated to $\varphi$, let $\kappa_{\varphi}: \C^{2n} \rightarrow \C^{2n}$ be the complex linear canonical transformation generated by $\varphi$, and let
\begin{align}
	\mathfrak{p}_0 = p_0 \circ \kappa_{\varphi}^{-1} \in \textrm{Hol}(\Lambda_{\Phi_0} + W),
\end{align}
where
\begin{align}
	\Lambda_{\Phi_0} = \set{\left(z, \frac{2}{i} \p_z \Phi_0(z) \right)}{z \in \C^n},
\end{align}
and $W$ is a small bounded open neighborhood of $(0,0)$ in $\C^{2n}$. Let 
\begin{align}
	H_{\mathfrak{p}_0} = \p_\zeta \mathfrak{p}_0 \cdot \p_z - \p_z \mathfrak{p}_0 \cdot \p_\zeta \in T^{1,0}(\Lambda_{\Phi_0}+W)
\end{align}
be the complex Hamilton vector field of $\mathfrak{p}_0$, and let
\begin{align}
	\kappa_t = \exp{\left(t H_{\mathfrak{p}_0} \right)}, \ \ t \in \C,
\end{align}
be the complex-time Hamilton flow of $\mathfrak{p}_0$, defined locally in $\Lambda_{\Phi_0}+W$ for sufficiently small complex times $t$. In this section, we shall be interested in the evolution of the subspace $\Lambda_{\Phi_0}$ by the flow $\kappa_t$ in a small, but fixed, neighborhood of the origin $0 \in \C^{2n}$. Using Hamilton-Jacobi theory, it is possible to obtain a real analytic generating function for this evolution. We describe this in detail below, following essentially the discussion in \cite{WF_Multiple}.

Let $\C^{2+2n} = \C^{2}_{t, \tau} \times \C^{2n}_{z, \zeta}$ be equipped with the standard complex symplectic form
\begin{align}
	\Omega = d \tau \wedge dt + \sigma \in \Lambda^{(2,0)}(\C^{2+2n}),
\end{align}
where $\sigma$ is as in (\ref{definition of the standard form}).
Let
\begin{align}
	G(t, \tau; z, \zeta) = \tau + \mathfrak{p}_0(z,\zeta), \ \ (t,\tau; z, \zeta) \in \textrm{neigh}((0,0; 0,0); \C^2 \times \C^{2n}),
\end{align}
and let
\begin{align}
	H^\Omega_G = \frac{\p}{\p t} + H_{\mathfrak{p}_0}
\end{align}
be the complex Hamilton vector field of $G$ with respect to $\Omega$, defined in an open neighborhood of $(0, 0; 0,0)$ in $\C^{2} \times \C^{2n}$. Consider the following real $2n$-dimensional submanifold of $\C^{2+2n}$:
\begin{align}
	\mathcal{L}_0 = \set{\left(0, - \mathfrak{p}_0 \left(z, \frac{2}{i} \p_z \Phi_0(z) \right); z, \frac{2}{i} \p_z \Phi_0(z) \right)}{z \in \textrm{neigh}(0; \C^n)}
\end{align}
Observe that $\mathcal{L}_0$ is isotropic with respect to the real symplectic form
\begin{align}
	\textrm{Im}(\Omega) = \textrm{Im}(d\tau \wedge dt) +\textrm{Im} \, \sigma
\end{align} 
on $\C^2 \times \C^{2n}$ and that
\begin{align}
	\mathcal{L}_0 \subset G^{-1}(0).
\end{align}
Let
\begin{align}
	\exp{(t H^\Omega_G)}, \ \ t \in \textrm{neigh}(0; \C),
\end{align}
be the complex-time Hamilton flow of $G$ with respect to the complex symplectic form $\Omega$, defined in a neighborhood of $(0,0; 0,0) \in \C^{2} \times \C^{2n}$ for sufficiently small complex times $t$. We observe that the complex-time flow generated by $H^\Omega_G$ is given explicitly by
\begin{align} \label{explicit form of flow on big space}
	\exp{\left(t H^\Omega_G \right)}(0,\tau; z, \zeta) = (t, \tau; \kappa_t(z,\zeta)), \ \ (t,\tau; z,\zeta) \in \textrm{neigh}((0,0; 0,0); \C^{2} \times \C^{2n}).
\end{align}
Let
\begin{align} \label{the flowout}
	\mathcal{L} := \bigcup_{t \in \textrm{neigh}(0; \C)} \exp{(tH^\Omega_G)} \left(\mathcal{L}_0 \right)
\end{align}
be the complex time flowout of the manifold $\mathcal{L}_0$ by $\exp{(tH^\Omega_G)}$. In view of (\ref{explicit form of flow on big space}), we have 
\begin{align} \label{explicit form of flowout}
	\mathcal{L} = \set{\left(t, -\mathfrak{p}_0 \left(z, \frac{2}{i} \p_z \Phi_0(z) \right); \kappa_t \left(z, \frac{2}{i} \p_z \Phi_0(z) \right) \right)}{t \in \textrm{neigh}(0; \C), \ z \in \textrm{neigh}(0; \C^n)},
\end{align}
when the flowout (\ref{the flowout}) is restricted to sufficiently small complex times. Also, from (\ref{formal definition of complex-time Hamilton flow}), (\ref{commuting vector fields}), and the identities (\ref{key identity for complex time flow}) and (\ref{Hamilton_field_identity_I_need}), we see that
\begin{align}
	\mathcal{L} = \bigcup_{t \in \textrm{neigh}(0; \C)} \exp{\left[(\textrm{Re} \, t) H^{\textrm{Im} \, \Omega}_{\textrm{Im} \, G} \right]} \exp{\left[(\textrm{Im} \, t) H^{\textrm{Im} \, \Omega}_{\textrm{Re} \, G} \right]}(\mathcal{L}_0),
\end{align}
where $\textrm{Re} \, G$ and $\textrm{Im} \, G$ are the real and imaginary parts of $G$, respectively. As is easily verified, the vector fields $H^{\textrm{Im} \, \Omega}_{\textrm{Im} \, G}$ and $H^{\textrm{Im} \, \Omega}_{\textrm{Re} \, G}$ are linearly independent and nowhere tangent to $\mathcal{L}_0$ in a neighborhood $(0,0; 0,0) \in \C^2 \times \C^{2n}$. Since $\textrm{dim} \, \mathcal{L}_0 = 2n$, it follows from Hamilton-Jacobi theory that the manifold $\mathcal{L}$ is Lagrangian for the real symplectic form $\textrm{Im}(\Omega)$ on $\C^{2} \times \C^{2n}$.

Let
\begin{align} \label{definition of the form gamma}
	\gamma = \textrm{Im}(\tau \, dt) + \textrm{Im}(\zeta \, dz) \in \Lambda^1(\C^2 \times \C^{2n}).
\end{align}
We observe that
\begin{align}
	d \gamma = \textrm{Im}(\Omega).
\end{align}
Because $\mathcal{L}$ is Lagrangian with respect to $\textrm{Im}(\Omega)$, we have
\begin{align} \label{exterior derivative vanishes}
	\left. d \gamma \right|_{\mathcal{L}} \equiv 0.
\end{align}
Let $\pi: \mathcal{L} \rightarrow \C \times \C^n$ be the restriction of the projection $(t,\tau; z,\zeta) \mapsto (t,z)$ to $\mathcal{L}$. Since $\Lambda_{\Phi_0}$ is transverse to the fiber $\{0\} \times \C^n$ in $\C^{2n}$, the differential of $\pi$ at $(0,0;0,0) \in \mathcal{L}$,
\begin{align}
	d_{(0,0; 0,0)} \pi: T_{(0,0; 0,0)} \mathcal{L} \rightarrow T_{(0;0)} (\C \times \C^n) \cong \C \times \C^n,
\end{align}
is an invertible linear transformation. Thus we may parametrize $\mathcal{L}$ by $(t,z) \in \C \times \C^n$ in a neighborhood of $(0,0;0,0)$ in $\mathcal{L}$. By (\ref{exterior derivative vanishes}) and Poincar\'{e}'s lemma, there exists $\Phi \in C^\omega(\textrm{neigh}(0; \C) \times \textrm{neigh}(0; \C^n); \R)$, unique up to an overall additive constant, such that
\begin{align} \label{generating function from poincare}
	\gamma(t,z) = -d \Phi(t,z), \ \ t \in \textrm{neigh}(0; \C), \ \ z \in \textrm{neigh}(0; \C^n).
\end{align}
From (\ref{definition of the form gamma}), we see that
\begin{align} \label{exterior derivative of gamma explicit}
	\gamma = \frac{1}{2i}(\tau \, dt - \overline{\tau} \, d \overline{t})  + \frac{1}{2i} \left(\zeta \, dz - \overline{\zeta} \, d\overline{z} \right), \ \ (t,\tau; z, \zeta) \in \C^2 \times \C^{2n}.
\end{align}
We also have
\begin{align} \label{exterior derivative of Phi}
\begin{split}
	-d \Phi(t,z) &= -\p_t \Phi(t,z) \, dt - \p_{\overline{t}} \Phi(t,z) \, d\overline{t} \\
	& \ \ \ - \p_z \Phi(t,z) \, dz - \p_{\overline{z}} \Phi(t,z) \, d \overline{z}, \ \ t \in \textrm{neigh}(0; \C), \ \ z \in \textrm{neigh}(0; \C^n).
\end{split}
\end{align}
From (\ref{generating function from poincare}), (\ref{exterior derivative of gamma explicit}), and (\ref{exterior derivative of Phi}), it follows that
\begin{align} \label{first part iff}
	(t, \tau; z, \zeta) \in \textrm{neigh}((0,0; 0,0); \mathcal{L})
\end{align}
if and only if
\begin{align} \label{conditions to be in submanifold}
t \in \textrm{neigh}(0; \C), \ \ z \in \textrm{neigh}(0; \C^n), \ \textrm{and} \ 
	\begin{cases}
		\tau = \frac{2}{i} \p_t \Phi(t,z), \\
		\zeta = \frac{2}{i} \p_z \Phi(t,z).
	\end{cases}
\end{align}
Thus there is $0<T \ll 1$ and $U = \textrm{neigh}(0; \C^n)$ such that
\begin{align}
	\kappa_t(\Lambda_{\Phi_0} \cap U \times U) \cap U \times U = \Lambda_{\Phi_t}, \ \ t \in D(0,T),
\end{align}
where $D(0,T)$ denotes the open disc in $\C$ with radius $T$ and center $0$, 
\begin{align}
	\Phi_t := \Phi(t, \cdot) \in C^\omega(U), \ \ t \in D(0,T),
\end{align}
and
\begin{align}
	\Lambda_{\Phi_t} = \set{\left(z, \frac{2}{i} \p_z \Phi_t(z)\right)}{z \in U}, \ \ t \in D(0,T).
\end{align}
Adjusting $\Phi$ by a real constant if necessary, we may assume that
\begin{align} \label{initial condition for Phi}
	\left. \Phi_t \right|_{t=0} = \Phi_0 \ \textrm{in} \ U.
\end{align}
Also, it is true that $\Phi_t$ is the solution of a natural eikonal equation in $D(0,T) \times U$. Indeed, since $\mathfrak{p}_0$ is invariant under the flow $\kappa_t$, we have
\begin{align} \label{invariance under flow implication}
	\mathfrak{p}_0 \left(z, \frac{2}{i} \p_z \Phi_0(z) \right) = \mathfrak{p}_0 \left(\kappa_t \left(z, \frac{2}{i} \p_z \Phi_0(z) \right) \right), \ \ t \in D(0,T), \ \ z \in U.
\end{align}
From (\ref{explicit form of flowout}), (\ref{first part iff}), (\ref{conditions to be in submanifold}), (\ref{initial condition for Phi}), and (\ref{invariance under flow implication}), we deduce that $\Phi_t$ solves the initial value problem
\begin{align} \label{eikonal initial value problem}
\begin{split}
	\begin{cases}
		2 \p_t \Phi_t(z) + i \mathfrak{p}_0 \left(z, \frac{2}{i} \p_z \Phi_t(z) \right) = 0, \ \ t \in D(0,T), \ \ z \in U, \\
		\left. \Phi_t \right|_{t=0} = \Phi_0 \ \textrm{in} \ U,
	\end{cases}
\end{split}
\end{align}
where $\p_t = \frac{1}{2} \left(\p_{\textrm{Re} \, t} - i \p_{\textrm{Im} \, t} \right)$. In the sequel, we shall refer to (\ref{eikonal initial value problem}) as the {\emph{complex-time eikonal equation}}.

Finally, we note that the function $\Phi_t \in C^\omega(U)$ is strictly plurisubharmonic in $U$ for every $t \in D(0,T)$, i.e.
\begin{align}
	\Phi''_{t, \overline{z} z}(z)>0, \ \ t \in D(0,T), \ \ z \in U.
\end{align}
To see why, we observe that since $\kappa_t$ preserves the complex symplectic form $\sigma$, the manifold $\Lambda_{\Phi_t}$ is $I$-Lagrangian and $R$-symplectic for every $t \in D(0,T)$. Parametrizing $\Lambda_{\Phi_t}$ by $z \in U$, we find that
\begin{align} \label{symplectic form in coordinates}
	\sigma|_{\Lambda_{\Phi_t}} &= \sum_{j=1}^n d \left(\frac{2}{i} \p_{z_j} \Phi(z) \right) \wedge dz_j = \frac{2}{i} \sum_{k,j=1}^n \frac{\p^2 \Phi_t}{\p \overline{z}_k \p z_j} \, d \overline{z}_k \wedge dz_j.
\end{align}
Because the two-form on the righthand side of (\ref{symplectic form in coordinates}) is real, we see that
\begin{align}
	\textrm{Re} \, \sigma|_{\Lambda_{\Phi_t}} = \frac{2}{i} \sum_{k,j=1}^n \frac{\p^2 \Phi_t}{\p \overline{z}_k \p z_j} \, d \overline{z}_k \wedge dz_j.
\end{align}
As $\left. \textrm{Re} \, \sigma \right|_{\Lambda_{\Phi_t}}$ is non-degenerate, we have
\begin{align}
	\textrm{det} \, \Phi''_{t, \overline{z} z}(z) \neq 0, \ \ t \in D(0,T), \ \ z \in U.
\end{align}
Because also $\Phi_{0, \overline{z} z} > 0$ and $\Phi''_{t, \overline{z} z}$ depends continuously on $t$, it must therefore be the case that
\begin{align}
	\Phi''_{t, \overline{z} z}(z)>0, \ \ t \in D(0,T), \ \ z \in U.
\end{align}

This discussion is summarized by the following proposition.

\begin{proposition}[\cite{WF_Multiple}] \label{generating function proposition}
	Let $p_0 \in S_{\textrm{Hol}}(m)$ be as in the statement of Theorem \ref{main_theorem}, and let $\varphi$ be an FBI phase function on $\C^{2n}$ with associated strictly plurisubharmonic weight $\Phi_0$ and complex canonical transformation $\kappa_\varphi: \C^{2n} \rightarrow \C^{2n}$. Let
	\begin{align}
		\Lambda_{\Phi_0} = \set{\left(z, \frac{2}{i} \p_z \Phi_0(z) \right)}{z \in \C^n},
	\end{align}
	let $\mathfrak{p}_0 := p_0 \circ \kappa_{\varphi}^{-1} \in \textrm{Hol}(\Lambda_{\Phi_0}+W)$, where $W$ is a suitably small open neighborhood of $0$ in $\C^{2n}$, and let $\kappa_t = \exp{(tH_{\mathfrak{p}_0})}$, $t \in \C$, be the complex-time Hamilton flow of $\mathfrak{p}_0$, defined in $\Lambda_{\Phi_0}+W$. Then there exists $0<T \ll 1$, $U = \textrm{neigh}(0; \C^n)$, and a unique $\Phi \in C^\omega(D(0,T) \times U; \R)$ such that
		\begin{align} \label{geometric interpretation of eikonal equation}
			\kappa_t(\Lambda_{\Phi_0} \cap U \times U) \cap U \times U = \Lambda_{\Phi_t}, \ \ t \in D(0,T),
		\end{align}
		and
		\begin{align}
			 \left. \Phi_t \right|_{t=0} = \Phi_0 \ \textrm{in} \ U,
		\end{align}
		where 
		\begin{align}
			\Phi_t := \Phi(t, \cdot) \in C^\omega(U; \R), \ \ t \in D(0,T),
		\end{align}
 		and
		\begin{align}
			 \Lambda_{\Phi_t} := \set{\left(z, \frac{2}{i} \p_z \Phi_t(z) \right)}{z \in U}, \ \ t \in D(0,T).
		\end{align}
		The function $\Phi_t$ is strictly plurisubharmonic in $U$ for each $t \in D(0,T)$, and $\Phi$ is a solution of the complex-time eikonal equation
		\begin{align} \label{complex_time_eikonal_equation}
			\begin{split}
			\begin{cases}
				2 \p_t \Phi_t(z) + i \mathfrak{p}_0 \left(z, \frac{2}{i} \p_z \Phi_t(z) \right) = 0, \ \ (t,z) \in D(0,T) \times U, \\
				\left. \Phi_t \right|_{t=0} = \Phi_0 \ \textrm{in} \ U,
			\end{cases}
			\end{split}
		\end{align}
		where $\p_t = \frac{1}{2} \left(\p_{\textrm{Re} \, t} - i \p_{\textrm{Im} \, t} \right)$.
\end{proposition}

Now we return our attention to eigenfunctions. Let $P = \textrm{Op}^w_h(p_0+hp_1)$, and $u \in L^2(\R^n)$ be as in the statement of Theorem \ref{main_theorem}. Let $\varphi$ be any FBI phase function on $\C^{2n}$ with associated FBI transformation $\mathcal{T}_\varphi$ and strictly plurisubharmonic weight $\Phi_0$, and let $\kappa_{\varphi}: \C^{2n} \rightarrow \C^{2n}$ be the complex linear canonical transformation generated by $\varphi$. For  $j=0,1$, let $\mathfrak{p}_j := p_j \circ \kappa_{\varphi}^{-1} \in \textrm{Hol}(\Lambda_{\Phi_0}+W)$, where $W$ is a suitably small bounded open neighborhood of $0$ in $\C^{2n}$. Let $\kappa_t = \exp{(tH_{\mathfrak{p}_0})}$, $t \in \C$ be the complex-time Hamilton flow of $\mathfrak{p}_0$, defined in $\Lambda_{\Phi_0}+W$. Suppose that $0<T \ll 1$, $U = \textrm{neigh}(0; \C^n)$, and $\Phi \in C^\omega(D(0,T) \times U; \R)$ are as in the conclusion of Proposition \ref{generating function proposition}.

For each $t \in D(0,T)$ and each open subset $U_0 \subset U$, let $H_{\Phi_t}(U_0)$ be the Hilbert space
\begin{align}
	H_{\Phi_t}(U_0) := L^2(U_0, e^{-2 \Phi_t(z)/h} \, L(dz)) \cap \textrm{Hol}(U_0),
\end{align}
equipped with the norm
\begin{align}
	\norm{v}^2_{L^2_{\Phi_t}(U_0)} := \int_{U_0} \abs{v(z)}^2 e^{-2 \Phi_t(z)/h} \, L(dz).
\end{align}
Here $\Phi_t = \Phi(t, \cdot) \in C^\omega(U)$, $t \in D(0,T)$. The main goal of this section is to show that there exists $\delta>0$, $0< T_0 < T$,  $0<C<\infty$, and $0<h_0 \le 1$, such that 
\begin{align} \label{dynamical_bound_first}
	\sup_{\substack{t \in D(0,T_0) \\ 0< h \le h_0}} \norm{\mathcal{T}_\varphi u}_{L^2_{\Phi_t}(\{\abs{z}<\delta\})} \le C.
\end{align}
To begin the proof, let
\begin{align}
	\mathcal{U}(t,z; h) := \mathcal{T}_\varphi u(h)(z), \ \ (t,z) \in \C \times \C^n, \ \ 0< h \le 1.
\end{align}
Let
\begin{align}
	\tilde{P} = \textrm{Op}^w_{\Phi_0, h}(\mathfrak{p}_0 + h\mathfrak{p}_1).
\end{align}
By Egorov's theorem,
\begin{align}
	 \tilde{P}\mathcal{U}(t,z;h) = 0, \ \ (t,z) \in \C \times \C^n, \ \ 0 < h \le 1.
\end{align}
Thus $\mathcal{U}$ is trivially a solution of the semiclassical Schr\"{o}dinger initial value problem
\begin{align} \label{complex_time_Schrodinger_IVP}
	\begin{split}
		\begin{cases}
			\left(h D_t + \tilde{P}\right) \mathcal{U}(t,z;h) = 0, \ (t,z) \in \C \times \C^n, \ \ 0< h \le 1, \\
			\mathcal{U}(0, z;h) = \mathcal{T}_\varphi u(h)(z), \ \ z \in \C^n, \ \ 0< h \le 1,
		\end{cases}
	\end{split}
\end{align}
where 
\begin{align}
D_t := \frac{1}{i} \p_t, \ \ \p_t := \frac{1}{2}\left(\p_{\textrm{Re} \, t} - i \p_{\textrm{Im} \, t} \right).
\end{align}
Let $\delta>0$ be small enough so that
\begin{align}
	\{\abs{z} < 5 \delta \} \subset \subset U.
\end{align}
Let $\chi \in C^\infty_0(\C^n; [0,1])$ be such that
\begin{align} \label{cutoff assumptions}
	\chi(z) = 1, \ \ \abs{z} \le 2 \delta, \ \ \textrm{and} \ \ \chi(z) = 0, \ \ \abs{z} \ge 3\delta,
\end{align}
and let $\widetilde{\chi} \in C^\infty_0(\C^n; [0,1])$ be such that
\begin{align} \label{tilde cutoff assumptions}
	\widetilde{\chi}(z) = 1, \ \ \abs{z} \le 4 \delta, \ \ \textrm{and} \ \ \widetilde{\chi}(z) = 0, \ \ \abs{z} \ge 5 \delta.
\end{align}
Let
\begin{align} \label{evolved_weights}
	\widetilde{\Phi}_t := (1-\widetilde{\chi}) \Phi_0 + \widetilde{\chi} \Phi_t, \ \ t \in D(0,T).
\end{align}
By construction, $\widetilde{\Phi}_t \in C^\infty(\C^n; \R)$ for each $t \in D(0,T)$, and we have
\begin{align}
	\widetilde{\Phi}_t(z) = \Phi_t(z), \ \ \abs{z} \le 2 \delta, \ \ t \in D(0,T),
\end{align}
and
\begin{align}
	\widetilde{\Phi}_t(z) = \Phi_0(z), \ \ \abs{z} \ge 5 \delta, \ \ t \in D(0,T).
\end{align}
To each $\widetilde{\Phi}_t$, we associate the Hilbert space
\begin{align}
	H_{\widetilde{\Phi}_t}(\C^n) = L^2(\C^n, e^{-2 \widetilde{\Phi}_t(z)/h} \, L(dz)) \cap \textrm{Hol}(\C^n),
\end{align}
which is equipped with the inner product
\begin{align}
	(v_1, v_2)_{L^2_{\widetilde{\Phi}_t}(\C^n)} := \int_{\C^n} v_1(z) \overline{v_2(z)} e^{-2 \widetilde{\Phi}_t(z)/h} \, L(dz), \ \ v_1, v_2 \in L^2_{\widetilde{\Phi}_t}(\C^n),
\end{align}
inherited from 
\begin{align}L^2_{\widetilde{\Phi}_t}(\C^n):=L^2(\C^n, e^{-2 \widetilde{\Phi}_t(z)/h} \, L(dz)).
\end{align}
Let
\begin{align}
	\norm{v}^2_{L^2_{\widetilde{\Phi}_t}(\C^n)} = \int_{\C^n} \abs{v(z)}^2 e^{-2 \widetilde{\Phi}_t(z)/h} \, L(dz), \ \ v \in L^2_{\widetilde{\Phi}_t}(\C^n),
\end{align}
denote the corresponding norm.

Consider the quantity
\begin{align}
	M_t := (\chi \mathcal{U}, \mathcal{U})_{L^2_{\widetilde{\Phi}_t}(\C^n)}, \ \ t \in D(0,T).
\end{align}
In view of (\ref{complex_time_Schrodinger_IVP}), we have
\begin{align}
	h D_t M_t &= - \left(\chi \tilde{P}\mathcal{U}, \mathcal{U} \right)_{L^2_{\widetilde{\Phi}_t}(\C^n)} - \frac{2}{i} \left(\chi \p_t \widetilde{\Phi}_t \mathcal{U}, \mathcal{U} \right)_{L^2_{\widetilde{\Phi}_t}(\C^n)}
\end{align}
Let $0< T_0 < T$ be small enough so that
\begin{align}
	\Lambda_{\widetilde{\Phi}_t} := \set{\left(z, \frac{2}{i} \p_z \widetilde{\Phi}_t(z) \right)}{z \in \C^n} \subset \Lambda_{\Phi_0} + W, \ \ t \in D(0,T_0)
\end{align}
and
\begin{align}
	\widetilde{\Phi}''_{t, \overline{z} z}(z) > 0, \ \ z \in \C^n, \ \ t \in D(0,T_0).
\end{align}
By taking $T_0$ smaller, we may make the quantity
\begin{align}
	\max_{k = 0,1,2} \sup_{t \in D(0,T_0)} \norm{\nabla^k \widetilde{\Phi}_t - \nabla^k \Phi_0}_{L^\infty(\C^n)}
\end{align}
as small as we wish. Thus, if $0<T_0<T$ is sufficiently small, the Hilbert spaces $L^2_{\widetilde{\Phi}_t}(\C^n)$ and $L^2_{\Phi_0}(\C^n)$ agree for every $t \in D(0,T_0)$ and have equivalent norms. We note, however, that these norms are not uniformly equivalent as $h \rightarrow 0^+$.
From Lemma 12.7 of  \cite{Lectures_on_Resonances}, we know that
\begin{align}
	\chi \tilde{P} = \mathcal{O}(1): H_{\widetilde{\Phi}_t}(\C^n) \rightarrow L^2_{\widetilde{\Phi}_t}(\C^n),
\end{align}
uniformly for $t \in D(0,T_0)$ and $0<h \ll 1$. By the quantization-multiplication theorem (Proposition 12.10 in \cite{Lectures_on_Resonances}), there is $0<h_0 \le 1$ such that
\begin{align}
\begin{split}
	&\left(\chi \tilde{P} \mathcal{U}, \mathcal{U} \right)_{L^2_{\widetilde{\Phi}_t}(\C^n)} \\ &= \int_{\C^n} \chi(z)\left[\mathfrak{p}_0 \left(z, \frac{2}{i} \p_z \widetilde{\Phi}_t(z) \right) + h \mathfrak{p}_1 \left(z, \frac{2}{i} \p_z \widetilde{\Phi}_t(z); h \right) \right] \abs{\mathcal{U}(t,z;h)}^2 e^{-2 \widetilde{\Phi}_t(z)/h} \, L(dz) \\
	& \ \ \ + \mathcal{O} \left( h \norm{\mathcal{U}}^2_{L^2_{\widetilde{\Phi}_t(\C^n)}} \right), \ \ t \in D(0,T_0), \ \ 0< h \le h_0.
\end{split}
\end{align}
Thus
\begin{align} \label{derivative of M_t}
	h D_t M_t = {\rm I}(t;h) + {\rm II}(t;h) + {\rm III}(t;h), \ \ t \in D(0,T_0), \ \ 0<h \le h_0,
\end{align}
where
\begin{align}
\begin{split}
	&{\rm I}(t;h) := i \int_{\C^n} \chi(z) \left[2 \p_t \widetilde{\Phi}_t(z) + i \mathfrak{p}_0 \left(z, \frac{2}{i} \p_z \widetilde{\Phi}_t(z) \right) \right] \abs{\mathcal{U}(t,z;h)}^2 e^{-2 \widetilde{\Phi}_t(z)/h} \, L(dz), \\
	&{\rm II}(t;h) := - h \int_{\C^n} \chi(z) \, \mathfrak{p}_1 \left(z, \frac{2}{i} \p_z \widetilde{\Phi}_t(z); h \right) \abs{\mathcal{U}(t,z;h)}^2 e^{-2 \widetilde{\Phi}_t(z)/h} \, L(dz),
\end{split}
\end{align}
and ${\rm III}(t;h)$ is such that
\begin{align} \label{property of III}
	{\rm III}(t;h) = \mathcal{O} \left( h \norm{\mathcal{U}}^2_{L^2_{\widetilde{\Phi}_t(\C^n)}} \right), \ \ t \in D(0,T_0), \ \ 0< h \le h_0.
\end{align}
As a consequence of (\ref{tilde cutoff assumptions}) and (\ref{evolved_weights}), we have
\begin{align} \label{identical weight}
	\widetilde{\Phi}_t(z) = \Phi_t(z), \ \ \p_z \widetilde{\Phi}_t(z) = \p_z \Phi_t(z), \ \ \abs{z} \le 3\delta, \ \ t \in D(0,T_0).
\end{align}
From (\ref{identical weight}), (\ref{cutoff assumptions}), and (\ref{complex_time_eikonal_equation}), we deduce
\begin{align} \label{I vanishes}
	{\rm I}(t;h) = i \int_{\abs{z} \le 3 \delta} \chi(z) \left[2 \p_t \Phi_t(z) + i \mathfrak{p}_0 \left(z, \frac{2}{i} \p_z \Phi_t(z) \right) \right] \abs{\mathcal{U}(t,z;h)}^2 e^{-2 \Phi_t(z)/h} \, L(dz) = 0
\end{align}
for all $t \in D(0,T_0)$ and $0 < h \le h_0$. Since $p_1 \in S_{\textrm{Hol}}(m)$ and $\kappa_{\varphi}: \C^{2n} \rightarrow \C^{2n}$ is linear, there exists $C>0$ and $N \in \R$ such that
\begin{align}
	\abs{\mathfrak{p}_1 \left(z, \frac{2}{i} \p_z \widetilde{\Phi}_t(z); h \right)} \le C \langle z \rangle^N, \ \ z \in \C^n, \ \ t \in D(0,T_0), \ \ 0< h \le h_0.
\end{align}
In particular, there is $C>0$ such that
\begin{align}
	\abs{\mathfrak{p}_1 \left(z, \frac{2}{i} \p_z \widetilde{\Phi}_t(z) ; h \right)} \le C, \ \ \abs{z} \le 3 \delta, \ \ t \in D(0,T_0), \ \ 0 < h \le h_0.
\end{align}
It follows that
\begin{align} \label{bound for II}
	{\rm II}(t;h) = \mathcal{O} \left( hM_t \right), \ \ t \in D(0,T_0), \ \ 0< h \le h_0.
\end{align}
Next, observe that
\begin{align} \label{first_bound_for_III}
\begin{split}
	\norm{\mathcal{U}}^2_{L^2_{\widetilde{\Phi}_t}(\C^n)} =  M_t + \int_{\C^n} (1-\chi(z)) \abs{\mathcal{T}_\varphi u(z)}^2 e^{-2 \widetilde{\Phi}_t(z)/h} \, L(dz), \ \ t \in D(0,T_0), \ \ 0<h \le h_0.
\end{split}
\end{align}
Since $\chi$ satisfies (\ref{cutoff assumptions}), we have
\begin{align}
	\int_{\C^n} (1-\chi(z)) \abs{\mathcal{T}_\varphi u(z)}^2 e^{-2 \widetilde{\Phi}_t(z)/h} \, L(dz) \le 
	\int_{\abs{z} \ge 2 \delta} \abs{\mathcal{T}_\varphi u(z)}^2 e^{-2 \widetilde{\Phi}_t(z)/h} \, L(dz),
\end{align}
for all $t \in D(0,T_0)$ and $0< h \le h_0$. From (\ref{tilde cutoff assumptions}) and (\ref{evolved_weights}), we see that
\begin{align}
	\widetilde{\Phi}_t(z) = \Phi_0(z), \ \ \abs{z} \ge 5 \delta, \ \ t \in D(0,T_0).
\end{align}
Hence
\begin{align}
	\int_{\abs{z} \ge 2 \delta} \abs{\mathcal{T}_\varphi u(z)}^2 e^{-2 \widetilde{\Phi}_t(z)/h} \, L(dz) = A(t;h) + B(t;h), \ \ t \in D(0,T_0), \ \ 0< h \le h_0,
\end{align}
where
\begin{align}
\begin{split}
	A(t;h) := \int_{2 \delta \le \abs{z} \le 5 \delta} \abs{\mathcal{T}_\varphi u(z)}^2 e^{-2 \widetilde{\Phi_t}(z)/h} \ L(dz), \ \ t \in D(0,T_0), \ \ 0<h \le h_0, \\
\end{split}
\end{align}
and
\begin{align}
	B(t;h) :=  \int_{\abs{z}>5 \delta} \abs{\mathcal{T}_\varphi u(z)}^2 e^{-2 \Phi_0(z)/h} \ L(dz), \ \ t \in D(0,T_0), \ \ 0< h \le h_0.
\end{align}
We note that $B(t;h)$ is independent of $t$. Taking $h_0$ smaller if necessary, we get from Corollary \ref{Wavefront_is_0} that there are constants $C, c>0$ such that
\begin{align}
	\abs{\mathcal{T}_\varphi u(z)} \le C e^{-c/h + \Phi_0(z)/h}, \ \ 2 \delta \le \abs{z} \le 5 \delta, \ \ 0< h \le h_0.
\end{align}
Since $\widetilde{\Phi}_t$ depends continuously on $t$ and $\widetilde{\Phi}_0 = \Phi_0$, we may assume, after taking $T_0$ smaller if necessary, that
\begin{align}
	\abs{\Phi_0(z) - \widetilde{\Phi}_t(z)} \le \frac{c}{2}, \ \ \abs{z} \le 5 \delta, \ \ t \in D(0,T_0).
\end{align}
It follows that
\begin{align}
	\abs{\mathcal{T}_\varphi u(z)}^2 e^{-2 \widetilde{\Phi}_t(z)/h} \le \abs{\mathcal{T}_\varphi u(z)}^2 e^{-2 \Phi_0(z)/h} e^{c/h} = \mathcal{O}(e^{-c/h}), \ \ 2 \delta \le \abs{z} \le 5 \delta, \ \ 0< h \le h_0.
\end{align}
Thus
\begin{align} \label{bound for A}
	A(t;h) = \mathcal{O}(e^{-c/h}), \ \ t \in D(0,T_0), \ \ 0<h \le h_0.
\end{align}
On the other hand, Proposition \ref{localization_lemma_FBI_side} gives
\begin{align} \label{bound for B}
	B(t;h) = \mathcal{O}(h^\infty), \ \ t \in D(0,T_0), \ \ 0< h \le h_0.
\end{align}
Putting (\ref{first_bound_for_III}), (\ref{bound for A}), and (\ref{bound for B}) together, we obtain
\begin{align} \label{bound for Upsilon}
	\norm{\mathcal{U}}^2_{L^2_{\widetilde{\Phi}_t}(\C^n)} = M_t + \mathcal{O}(h^\infty), \ \ t \in D(0,T_0), \ \ 0< h \le h_0.
\end{align}
Then, from (\ref{property of III}) and (\ref{bound for Upsilon}), we get
\begin{align} \label{bound for III}
	{\rm III}(t;h) = \mathcal{O}(hM_t) + \mathcal{O}(h^\infty), \ \ t \in D(0,T_0), \ \ 0< h \le h_0.
\end{align}
Combining (\ref{derivative of M_t}), (\ref{I vanishes}), (\ref{first_bound_for_III}), and (\ref{bound for III}) therefore yields
\begin{align}
	h D_t M_t = \mathcal{O}(h M_t) + \mathcal{O}(h^\infty), \ \ t \in D(0,T_0), \ \ 0< h \le h_0.
\end{align}
We conclude that
\begin{align} \label{good bound for derivative}
	\p_t M_t = \mathcal{O}(M_t) + \mathcal{O}(h^\infty), \ \ t \in D(0,T_0), \ \ 0< h \le h_0.
\end{align}

Let $\alpha \in \{\abs{z} = 1 \}$ be arbitrary. The function
\begin{align}
	(-T_0, T_0) \ni s \mapsto M_{\alpha s}
\end{align}
is smooth and real-valued. By the chain rule,
\begin{align} \label{derivative along ray}
	\frac{d}{ds} M_{\alpha s} &= \left. \p_t M_t \right|_{t = \alpha s} \alpha + \left. \p_{\overline{t}} M_{t} \right|_{t = \alpha s} \overline{\alpha} = 2 \textrm{Re} \, \left( \left. \p_t M_{t} \right|_{t = \alpha s} \alpha \right), \ \ s \in (-T_0, T_0).
\end{align}
Bounding the righthand side of (\ref{derivative along ray}) using (\ref{good bound for derivative}), we find that for any $N>0$, there is $C>0$, independent of $\alpha$, such that
\begin{align}
	\frac{d}{ds} M_{\alpha s} \le C M_{\alpha s} + C h^N, \ \ s \in (-T_0, T_0), \ \ 0< h \le h_0.
\end{align}
By Gronwall's inequality, for any $N>0$, there is $C>0$, independent of $\alpha$, such that
\begin{align} \label{bound in positive time direction}
	M_{\alpha s} \le C M_0 + C h^N, \ \ s \in [0,T_0), \ \ 0< h \le h_0.
\end{align}
Since
\begin{align}
	M_0 \le \norm{\mathcal{T}_\varphi u}^2_{L^2_{\Phi_0}(\C^n)} = 1, \ \ 0< h \le h_0,	
\end{align}
and (\ref{bound in positive time direction}) holds for any $\alpha \in \{\abs{z}=1\}$, we conclude, after taking $h_0$ smaller if necessary, that
\begin{align} \label{M_t finite}
	M_t \le C, \ \ t \in D(0,T_0), \ \ 0< h \le h_0,
\end{align}
for some constant $0< C < \infty$. In view of (\ref{cutoff assumptions}), (\ref{tilde cutoff assumptions}), and (\ref{M_t finite}), we have
\begin{align}
	\norm{\mathcal{T}_\varphi u}_{L^2_{\Phi_t}(\{ \abs{z} < \delta \})} \le M^{1/2}_t < C^{1/2}, \ \ t \in D(0,T_0), \ \ 0< h \le h_0.
\end{align}
Therefore it is true that (\ref{dynamical_bound_first}) holds for some $\delta>0$ and $0<T_0 < T$.  We have proved the following proposition.

\begin{proposition} \label{locally in evolved spaces}
	Let $P = \textrm{Op}^w_h(p_0+hp_1)$ and $u \in L^2(\R^n)$ be as in the statement of Theorem \ref{main_theorem}. Let $\varphi$ be any FBI phase function on $\C^{2n}$ with associated FBI transformation $\mathcal{T}_\varphi$ and strictly plurisubharmonic weight $\Phi_0$, and let $\kappa_{\varphi}: \C^{2n} \rightarrow \C^{2n}$ be the complex linear canonical transformation generated by $\varphi$. Let $\mathfrak{p}_j := p_j \circ \kappa_{\varphi}^{-1} \in \textrm{Hol}(\Lambda_{\Phi_0}+W)$, $j=0,1$, where $W$ is a suitably small bounded open neighborhood of $0$ in $\C^{2n}$, and let $\kappa_t = \exp{(tH_{\mathfrak{p}_0})}$, $t \in \C$ be the complex-time Hamilton flow of $\mathfrak{p}_0$, defined in $\Lambda_{\Phi_0}+W$.  Suppose that $0<T \ll 1$, $U = \textrm{neigh}(0; \C^n)$, and $\Phi \in C^\omega(D(0,T) \times U; \R)$ are such that
	\begin{align} \label{evolution of manifold in box}
		\kappa_t(\Lambda_{\Phi_0} \cap U \times U) \cap U \times U = \Lambda_{\Phi_t}, \ \ t \in D(0,T),
	\end{align}
	and
	\begin{align}
		\left. \Phi_t \right|_{t=0} = \Phi_0 \ \textrm{in} \ U,
	\end{align}
	where 
	\begin{align}
	\Phi_t = \Phi(t,\cdot) \in C^\omega(U; \R), \ \ t \in D(0,T), 
	\end{align}
	and
	\begin{align}
		\Lambda_{\Phi_t} = \set{\left(z, \frac{2}{i} \p_z \Phi_t(z) \right)}{z \in U}, \ \ t \in D(0,T).
	\end{align}
	Then there are constants $\delta>0$, $0<T_0 < T$, $0<C<\infty$, and $0< h_0 \le 1$ such that
		\begin{align} \label{dynamical_bound}
			\sup_{\substack{t \in D(0,T_0) \\ 0< h \le h_0}} \norm{\mathcal{T}_\varphi u}_{L^2_{\Phi_t}(\{\abs{z}<\delta\})} \le C.
		\end{align}
\end{proposition}

\section{The Conclusion of the Proof of Theorem \ref{main_theorem}}

Let $p_0$ be as in the statement of Theorem \ref{main_theorem}. We start this section by showing that there exists an FBI phase function $\varphi$ on $\C^{2n}$ whose associated complex linear canonical transformation $\kappa_{\varphi}: \C^{2n} \rightarrow \C^{2n}$ is such that the quadratic approximation $\mathfrak{q}$ to $\mathfrak{p}_0 := p \circ \kappa_{\varphi}^{-1}$ at the origin $0 \in \C^{2n}$ has the convenient form
\begin{align} \label{good form of quadratic approx}
	\mathfrak{q}(z,\zeta) = M z \cdot \zeta, \ \ (z,\zeta) \in \C^{2n}.
\end{align}
where $M$ is a suitable complex $n \times n$ matrix.

Let $q$ be the quadratic approximation to $p_0$ at $0 \in \C^{2n}$. Since $q$ is elliptic along its singular space $S$, it follows from Proposition 2.0.1 of \cite{QuadraticOperators} that there exists a symplectic splitting of the coordinates of $\R^{2n}$,
\begin{align} \label{splitting of real coordinates}
\begin{split}
	\R^{2n} = \R^{2n'} \times \R^{2n''}, \ \ n = n'+n'', \ \ 0 \le n', n'' \le n, \ \ (x,\xi) = (x', \xi'; x'', \xi''),
\end{split}
\end{align}
and an $\R$-linear canonical transformation $\kappa_{\Re}: \R^{2n} \rightarrow \R^{2n}$ such that
\begin{align} \label{pullback q by real can trans}
	(q \circ \kappa_{\Re}^{-1})(x,\xi) = q_1(x',\xi') + q_2(x'',\xi''), \ \ (x,\xi) \in \R^{2n},
\end{align}
where $q_1$ is a complex-valued quadratic form on $\R^{2n'}$ having non-negative real part $\textrm{Re} \, q_1 \ge 0$ and trivial singular space, $S_1 = \{0\}$, and $q_2$ is a purely imaginary quadratic form on $\R^{2n''}$ of the form
\begin{align} \label{form of q_2}
	q_2(x'',\xi'') = i \epsilon \sum_{j=1}^{n''} \lambda_j (\abs{x''_j}^2 + \abs{\xi''_j}^2), \ \ (x'',\xi'') \in \R^{2n''},
\end{align}
where $\epsilon \in \{\pm 1\}$ and $\lambda_j>0$ for all $1 \le j \le n''$. We will show that there exist FBI phase functions $\varphi_1$ and $\varphi_2$ on $\C^{2n'}$ and $\C^{2n''}$, respectively, with associated complex linear canonical transformations $\kappa_{\varphi_1}: \C^{2n'} \rightarrow \C^{2n'}$ and $\kappa_{\varphi_2}: \C^{2n''} \rightarrow \C^{2n''}$ such that
\begin{align} \label{q_1 good normal form}
	\mathfrak{q}_1(z', \zeta') := \left(q_1 \circ \kappa_{\varphi_1}^{-1}\right)(z',\zeta') = M_1 z' \cdot \zeta', \ \ (z',\zeta') \in \C^{2n'},
\end{align}
and
\begin{align} \label{q_2 good normal form}
	\mathfrak{q}_2(z'', \zeta'') := \left(q_2 \circ \kappa_{\varphi_2}^{-1}\right)(z'',\zeta'') = M_2 z'' \cdot \zeta'', \ \ (z'', \zeta'') \in \C^{2n''},
\end{align}
for some $M_1 \in \textrm{Mat}_{n' \times n'}(\C)$ and $M_2 \in \textrm{Mat}_{n'' \times n''}(\C)$.

We begin by proving the existence of an FBI phase function $\varphi_1$ on $\C^{2n'}$ such that (\ref{q_1 good normal form}) holds. For this, we will closely follow the presentation of Section 2 of \cite{Lp_bounds}. The method we employ originates from the works \cite{resolvent_esimates_elliptic} and \cite{SpectralProjectionsAndResolventBounds}. Let $\C^{2n'} = \C^{n'}_{z'} \times \C^{n'}_{\zeta'}$ be equipped with the standard complex symplectic form $\sigma^{(1)} = d\zeta' \wedge dz' \in \Lambda^{(2,0)}(\C^{2n'})$. Let $F_1$ be the Hamilton matrix of $q_1$. Since the singular space $S_1$ of $q_1$ is trivial, $S_1 = \{0\}$, we know from the work \cite{QuadraticOperators} that $F_1$ possesses no real eigenvalues. It follows that
\begin{align}
	\#\set{\lambda \in \textrm{Spec}(F_1)}{\textrm{Im} \, \lambda >0} = \# \set{\lambda \in \textrm{Spec}(F_1)}{\textrm{Im} \, \lambda < 0},
\end{align}
when counting algebraic multiplicities. For $\lambda \in \textrm{Spec}(F_1)$, let us denote the generalized eigenspace of $F_1$ corresponding to $\lambda \in \textrm{Spec}(F_1)$ by
\begin{align}
	V_\lambda = \textrm{ker}((F_1 - \lambda I)^{2n}) \subset \C^{2n'}.
\end{align}
Next, let
\begin{align}
	\Lambda^+ := \bigoplus_{\substack{\lambda \in \textrm{Spec}(F_1) \\ \textrm{Im} \, \lambda >0}} V_\lambda, \ \ \Lambda^{-} := \bigoplus_{\substack{\lambda \in \textrm{Spec}(F_1) \\ \textrm{Im} \, \lambda < 0}} V_\lambda,
\end{align}
denote the stable outgoing and stable incoming manifolds for the quadratic form $-i q_1$, respectively. From Proposition 2.1 of the work \cite{SpectralProjectionsAndResolventBounds}, we know that $\Lambda^+$ is a strictly positive $\C$-Lagrangian subspace of $\C^{2n'}$ and that $\Lambda^{-}$ is a strictly negative $\C$-Lagrangian subspace of $\C^{2n'}$. For background on positive and negative $\C$-Lagrangian subspaces of $\C^{2n}$, see \cite{Minicourse} and \cite{ComplexFIOs}. From the discussion on pages 488-489 of \cite{Minicourse}, we may thus conclude that there is an FBI phase function $\varphi_1$ on $\C^{2n'}$ such that the complex linear canonical transformation $\kappa_{\varphi_1}: \C^{2n'} \rightarrow \C^{2n'}$ generated by $\varphi_1$ satisfies
\begin{align} \label{image of plus minus lagrangian}
	\kappa_{\varphi_1}(\Lambda^+) = \set{(z',0)}{z' \in \C^{n'}}, \ \ \kappa_{\varphi_1}(\Lambda^-) = \set{(0,\zeta')}{\zeta' \in \C^{n'}}.
\end{align}
Let
\begin{align}
	\Phi^{(1)}_0(z') := \textrm{max}_{y' \in \R^{n'}} \left(-\textrm{Im} \, \varphi_1(z',y')\right), \ \ z' \in \C^{n'},
\end{align}
be the strictly plurisubharmonic weight on $\C^{n'}$ associated to $\varphi_1$, and let
\begin{align}
	\Lambda_{\Phi^{(1)}_0} = \set{\left(z, \frac{2}{i} \p_{z'} \Phi^{(1)}_0(z') \right)}{z' \in \C^{n'}}.
\end{align}
The subspace $\Lambda_{\Phi^{(1)}_0}$ of $\C^{2n'}$ is $I$-Lagrangian and $R$-symplectic for the complex symplectic form $\sigma^{(1)}$, and we have
\begin{align} \label{first I lagrangian}
	\kappa_{\varphi_1}(\R^{2n'}) = \Lambda_{\Phi^{(1)}_0}.
\end{align}
Since $\Lambda^+$ is a strictly positive $\C$-Lagrangian subspace of $\C^{2n'}$, the base $\set{(z',0)}{z' \in \C^{n'}}$ is a $\C$-Lagrangian subspace of $\C^{2n'}$ that is strictly positive relative to  $\Lambda_{\Phi^{(1)}_0}$. As explained in Chapter 11 of \cite{AnalyticMicrolocal_Analysis}, we may therefore conclude that the strictly plurisubharmonic quadratic form $\Phi^{(1)}_0$ on $\C^{n'}$ is in fact strictly convex.

Let
\begin{align} \label{pulled back q_1}
	\mathfrak{q}_1 := q \circ \kappa_{\varphi_1}^{-1} \in \textrm{Hol}(\C^{2n'}).
\end{align}
Since $V_\lambda$ is invariant under $F_1$ for every $\lambda \in \textrm{Spec}(F_1)$, we know that $\Lambda^+$ and $\Lambda^-$ are both invariant under $F_1$. Since also $\Lambda^+$ and $\Lambda^-$ are Lagrangian with respect to $\sigma^{(1)}$, we have
\begin{align} \label{q_1 vanishes on plus minus lagrangian}
	q_1(Z') = \sigma^{(1)}(Z', FZ') = 0, \ \ Z' \in \Lambda^+ \cup \Lambda^-.
\end{align}
From (\ref{image of plus minus lagrangian}), (\ref{pulled back q_1}), and (\ref{q_1 vanishes on plus minus lagrangian}), we deduce that
\begin{align} \label{pullback q_1}
	\mathfrak{q}_1(z', \zeta') = M_1 z' \cdot \zeta', \ \ (z', \zeta') \in \C^{2n'},
\end{align}
for some $M_1 \in \textrm{Mat}_{n' \times n'}(\C)$. Thus $\varphi_1$ is an FBI phase function on $\C^{2n'}$ satisfying (\ref{q_1 good normal form}).

Now we prove that there is an FBI phase function $\varphi_2$ on $\C^{2n''}$ satisfying (\ref{q_2 good normal form}). Let $\C^{2n''} = \C^{n'}_{z'} \times \C^{n'}_{\zeta'}$ be equipped with the complex symplectic form $\sigma^{(2)} = d\zeta'' \wedge dz'' \in \Lambda^{(2,0)}(\C^{2n''})$. Let
\begin{align} \label{standard phase}
	\varphi_2(z'', y'') = \frac{i}{2} (z'')^2 - i \sqrt{2} z'' \cdot y'' + \frac{i}{2} y'' \cdot y'', \ \ (z'',y'') \in \C^{2n''},
\end{align}
be the `standard FBI phase' on $\C^{2n''}$ (see the discussion on pages 304-306 of \cite{SemiclassicalAnalysis}). A straightforward computation shows that the complex linear canonical transformation $\kappa_{\varphi_2}: \C^{2n''} \rightarrow \C^{2n''}$ generated by $\varphi_1$ is
\begin{align}
	\kappa_{\varphi_2}(x'',\xi'') = \frac{1}{\sqrt{2}} (x''-i \xi'', \xi'' - i x''), \ \ (x'',\xi'') \in \C^{2n''}.
\end{align}
The inverse of $\kappa_{\varphi_2}$ is then easily seen to be
\begin{align} \label{inverse of can trans}
	\kappa_{\varphi_2}^{-1}(z'', \zeta'') = \frac{1}{\sqrt{2}}(z''+i \zeta'', \zeta'' + i z''), \ \ (z'', \zeta'') \in \C^{2n''}.
\end{align}
A direct computation using (\ref{standard phase}) shows that the strictly plurisubharmonic weight associated to $\varphi_2$ is
\begin{align} \label{second psh weight}
	\Phi^{(2)}_0(z'') := \frac{\abs{z''}^2}{2}, \ \ z'' \in \C^{n''},
\end{align}
and we have
\begin{align} \label{image second i lagrangian}
	\kappa_{\varphi_2}(\R^{2n''}) = \Lambda_{\Phi^{(2)}_0} := \set{\left(z'', \frac{\overline{z''}}{i} \right)}{z'' \in \C^{n''}}.
\end{align}
Using (\ref{form of q_2}) and (\ref{inverse of can trans}), we see that
\begin{align} \label{the pullback of q_2}
\begin{split}
	\mathfrak{q}_2(z'', \zeta'') := (q_2 \circ \kappa_{\varphi_2}^{-1})(z'', \zeta'') &= i \epsilon \sum_{j=1}^{n''} \lambda_j \left(\left(\frac{z_j''+i\zeta''_j}{\sqrt{2}}\right)^2 + \left(\frac{\zeta''_j + i z''_j}{\sqrt{2}} \right)^2 \right) \\
	&= - 2 \epsilon \sum_{j=1}^{n''} \lambda_j z''_j \zeta''_j, \ \ (z'', \zeta'') \in \C^{2n''}.
\end{split}
\end{align}
Hence
\begin{align} \label{pullback q_2}
\mathfrak{q}_2(z'', \zeta'') = M_2 z'' \cdot \zeta'', \ \ (z'', \zeta'') \in \C^{2n''},	
\end{align}
where $M_2 \in \textrm{Mat}_{n'' \times n''}(\C)$ is the diagonal matrix
\begin{align} \label{pullback matrix}
	M_2 = \textrm{diag}(-2 \epsilon \lambda_1, \ldots, -2 \epsilon \lambda_{n''}).
\end{align}

Write
\begin{align}
	\C^{2n} = \C^{2n'} \times \C^{2n''}, \ \ (z,\zeta) = (z', \zeta'; z'', \zeta'').
\end{align}
For $j=0,1$, let
\begin{align}
	\mathfrak{p}_j := p_j \circ (\kappa^\C_{\Re})^{-1} \circ (\kappa_{\varphi_1}^{-1} \times \kappa_{\varphi_2}^{-1}) \in \textrm{Hol}\left(\Lambda_{\Phi^{(1)}_0} \times \Lambda_{\Phi^{(2)}_0} + W_1 \times W_2 \right),
\end{align} 
where $W_1$ and $W_2$ are sufficiently small bounded open neighborhoods of $0$ in $\C^{2n'}$ and $\C^{2n''}$ respectively, and $\kappa^\C_{\Re}: \C^{2n} \rightarrow \C^{2n}$ is the complexification of $\kappa_{\Re}: \R^{2n} \rightarrow \R^{2n}$. From (\ref{pullback q by real can trans}), (\ref{pullback q_1}), (\ref{the pullback of q_2}), and (\ref{pullback q_2}), we deduce that the quadratic approximation $\mathfrak{q}$ to $\mathfrak{p}_0$ at $0 \in \C^{2n}$ is given by
\begin{align}
	\mathfrak{q}(z,\zeta) := M z \cdot \zeta, \ \ (z,\zeta) \in \C^{2n},
\end{align}
for the matrix
\begin{align} \label{big matrix}
	M =
	\begin{pmatrix}
		M_1 & 0 \\
		0 & M_2
	\end{pmatrix} \in \textrm{Mat}_{n \times n}(\C).
\end{align}
Let
\begin{align} \label{definition of kappa}
	\kappa = (\kappa_{\varphi_1} \times \kappa_{\varphi_2}) \circ \kappa^\C_{\Re} : \C^{2n} \rightarrow \C^{2n}.
\end{align}
We would like to introduce a metaplectic FBI transform $\mathcal{T}_\varphi$ on $\R^n$ whose underlying complex linear canonical transformation is $\kappa$. The existence of such an FBI transform is guaranteed by the following proposition.

\begin{proposition} \label{good FBI phase function proposition}

 There exists a unique FBI phase function $\varphi$ on $\C^{2n}$ whose associated complex linear canonical transformation $\kappa_\varphi$ is precisely $\kappa$,
\begin{align} \label{canonical transformation as desired}
	\kappa_{\varphi} = \kappa.
\end{align}

\end{proposition}

\begin{proof}
Let
\begin{align} \label{definition of the good weight}
	\Phi_0(z) := \Phi^{(1)}_0(z') + \Phi^{(2)}_0(z''), \ \ z = (z', z'') \in \C^n = \C^{n'} \times \C^{n''}.
\end{align}
Since $\Phi^{(1)}_0$ and $\Phi^{(2)}_0$ are strictly plurisubharmonic quadratic forms on $\C^{n'}$ and $\C^{n''}$, respectively, the function $\Phi_0$ is a strictly plurisubharmonic quadratic form on $\C^{n}$. Let
\begin{align} \label{IR subspace for the good weight}
	\Lambda_{\Phi_0} := \set{\left(z, \frac{2}{i} \p_z \Phi_0(z) \right)}{z \in \C^n}.
\end{align}
Thus
\begin{align} \label{decomposition Lambda phi}
	\Lambda_{\Phi_0} = \Lambda_{\Phi^{(1)}_0} \times \Lambda_{\Phi^{(2)}_0}.
\end{align}
From (\ref{first I lagrangian}), (\ref{image second i lagrangian}), and (\ref{decomposition Lambda phi}), we deduce that
\begin{align}
	(\kappa_{\varphi_1} \times \kappa_{\varphi_2})(\R^{2n}) = \Lambda_{\Phi_0}.
\end{align}
Since we also have
\begin{align}
\kappa^\C_{\Re}(\R^{2n}) = \R^{2n},
\end{align}
it follows that 
\begin{align}
	\kappa(\R^{2n}) = \Lambda_{\Phi_0}.
\end{align}
The existence and uniqueness of an FBI phase function $\varphi$ on $\C^{2n}$ such that (\ref{canonical transformation as desired}) holds now follows from well-known arguments. For the details, the reader may consult, for instance, the discussion on pages 393-394 of \cite{WF_Multiple}.
\end{proof}

Having shown that $\kappa$ is generated by an FBI phase function $\varphi$ on $\C^{2n}$, we may now study the evolution of the subspace $\Lambda_{\Phi_0}$ introduced in (\ref{IR subspace for the good weight}) by the complex-time Hamilton flow generated by $\mathfrak{p}_0 := p_0 \circ \kappa^{-1}$.

\begin{proposition} \label{good complex time proposition}
	Let $p_0$ be as in the statement of Theorem \ref{main_theorem}, let $\varphi$ be as in the statement of Proposition \ref{good FBI phase function proposition}, let $\Phi_0$ be the strictly plurisubharmonic weight associated to $\varphi$, let $\kappa_\varphi: \C^{2n} \rightarrow \C^{2n}$ be the complex linear canonical transformation generated by $\varphi$, let $\mathfrak{p}_0 = p_0 \circ \kappa_\varphi^{-1} \in \textrm{Hol}(\Lambda_{\Phi_0}+W)$, where $W$ is a sufficiently small open bounded neighborhood of $0$ in $\C^n$, and let $\kappa_t = \exp{(tH_{\mathfrak{p}_0})}$, $t \in \C$, be the complex-time Hamilton flow of $\mathfrak{p}_0$ in $\Lambda_{\Phi_0}+W$. Suppose that $0<T \ll 1$, $U = \textrm{neigh}(0; \C^n)$, and $\Phi \in C^\omega(D(0,T) \times U; \R)$ are such that
	\begin{align} \label{evolution of manifold in box}
		\kappa_t(\Lambda_{\Phi_0} \cap U \times U) \cap U \times U = \Lambda_{\Phi_t}, \ \ t \in D(0,T),
	\end{align}
	and
	\begin{align}
		\left. \Phi_t \right|_{t=0} = \Phi_0 \ \textrm{in} \ U,
	\end{align}
	where 
	\begin{align} 
		\Phi_t = \Phi(t,\cdot) \in C^\omega(U; \R), \ \ t \in D(0,T),
	\end{align}
 and
	\begin{align}
		\Lambda_{\Phi_t} = \set{\left(z, \frac{2}{i} \p_z \Phi_t(z) \right)}{z \in U}, \ \ t \in D(0,T).
	\end{align}
	Then, for every $0<T_0<T$, there exists $t_0 \in D(0,T_0) \backslash \{0\}$, $\delta>0$, and $c>0$ such that
	\begin{align}
		\Phi_0(z) - \Phi_{t_0}(z) \ge c \abs{z}^2, \ \ \abs{z} < \delta.
	\end{align}
\end{proposition}

\begin{proof}

By Proposition \ref{generating function proposition}, the function $\Phi$ satisfies the complex-time eikonal equation
		\begin{align} \label{local eikonal equation}
			\begin{split}
			\begin{cases}
				2 \p_t \Phi_t(z) + i \mathfrak{p}_0 \left(z, \frac{2}{i} \p_z \Phi_t(z) \right) = 0, \ \ (t,z) \in D(0,T) \times U, \\
				\left. \Phi_t \right|_{t=0} = \Phi_0 \ \textrm{in} \ U,
			\end{cases}
			\end{split}
		\end{align}
where $\p_t = \frac{1}{2} \left(\p_{\textrm{Re} \, t} - i \p_{\textrm{Im} \, t} \right)$. Since $p_0$ vanishes to second order at $0 \in \C^{2n}$ and $\kappa_{\varphi}(0) = 0$, the symbol $\mathfrak{p}_0$ must also vanish to second order at $0 \in \C^{2n}$. Thus the complex Hamilton vector field of $H_{\mathfrak{p}_0}$ of $\mathfrak{p}_0$ vanishes at $0 \in \C^{2n}$,
\begin{align}
	\left. H_{\mathfrak{p}_0} \right|_{0} = 0.
\end{align}
Therefore
\begin{align} \label{origin is fixed by complex family}
	\kappa_t(0) = 0, \ \ t \in D(0,T).
\end{align}
This observation, in conjunction with (\ref{evolution of manifold in box}), implies that
\begin{align}
	\p_z \Phi_t(0) = 0, \ \ t \in D(0,T).
\end{align}
Hence, when we take $z=0$ in (\ref{local eikonal equation}), we obtain that
\begin{align}
	\p_t \Phi_t(0) = 0, \ \ t \in D(0,T).
\end{align}
Because $\Phi_0$ is quadratic, we have $\Phi_0(0) = 0$, and so
\begin{align}
	\Phi_t(0) = 0, \ \ t \in D(0,T).
\end{align}
We conclude that
\begin{align}
	\Phi_t \ \textrm{vanishes to $2$nd order at $z=0$ for all $t \in D(0,T)$}.
\end{align}

Let $\Xi \in C^\omega(D(0,T) \times \C^n; \R)$ be the unique analytic function on $D(0,T) \times \C^n$ such that $\Xi_t:=\Xi(t,\cdot)$ is the quadratic approximation to $\Phi_t$ at $z = 0$ in $\C^n$ for each $t \in D(0,T)$, i.e. $\Xi_t$ is the unique real quadratic form on $\C^n$ such that
\begin{align} \label{quadratic approximation to the evolved weight}
	\Phi_t(z) = \Xi_t(z) + \mathcal{O}(\abs{z}^3), \ \ \abs{z} \rightarrow 0^+,
\end{align}
for each fixed $t \in D(0,T)$. Note that, by Proposition \ref{generating function proposition}, $\Phi_t$ is strictly plurisubharmonic in $U$ for each $t \in D(0,T)$, and hence $\Xi_t$ is a strictly plurisubharmonic quadratic form on $\C^n$ for each $t \in D(0,T)$, i.e.
\begin{align}
	\Xi''_{t,\overline{z} z} > 0, \ \ t \in D(0,T).
\end{align}
Taylor expanding (\ref{local eikonal equation}) to second order about the origin $z=0$ shows that $\Xi$ is the unique solution of the quadratic complex-time eikonal equation
\begin{align} \label{quadratic complex time eikonal equation}
\begin{split}
	\begin{cases}
		2 \p_t \Xi_t(z) + i \mathfrak{q} \left(z, \frac{2}{i} \p_z \Xi_t(z) \right) = 0, \ \ (t,z) \in D(0,T) \times \C^n, \\ 
		\left. \Xi_t \right|_{t=0} = \Phi_0 \ \textrm{on} \ \C^n,
	\end{cases}
\end{split}
\end{align}
where $\mathfrak{q}$ is the quadratic approximation to $\mathfrak{p}_0$ at $0 \in \C^{2n}$. Let us make the following splitting of coordinates in $\C^n$:
\begin{align}
	\C^{n} = \C^{n'} \times \C^{n''}, \ \ z = (z',z''),
\end{align}
where $0 \le n', n'' \le n$ are as in (\ref{splitting of real coordinates}). We search for a solution to (\ref{quadratic complex time eikonal equation}) of the form
\begin{align} \label{ansatz for quadratic eikonal equation}
	\Xi(t,z) = \Xi^{(1)}(t, z') + \Xi^{(2)}(t, z''), \ \ (t,z) \in D(0,T) \times \C^n,
\end{align}
where $\Xi^{(1)} \in C^\omega(D(0,T) \times \C^{n'}; \R)$, $\Xi^{(2)} \in C^\omega(D(0,T) \times \C^{n''}; \R)$, and $\Xi^{(1)}_t := \Xi^{(1)}(t, \cdot)$ and $\Xi^{(2)}_t := \Xi^{(2)}(t, \cdot)$ are strictly plurisubharmonic quadratic forms on $\C^{n'}$ and $\C^{n''}$ for each $t \in D(0,T)$, respectively. Since $\mathfrak{q}$ is of the form (\ref{good form of quadratic approx}), where the matrix $M$ is given by (\ref{big matrix}), and since the strictly plurisubharmonic weight $\Phi_0$ has the form (\ref{definition of the good weight}), we see that (\ref{ansatz for quadratic eikonal equation}) will be a solution of the problem (\ref{quadratic complex time eikonal equation}) provided $\Xi^{(1)}_t$ and $\Xi^{(2)}_t$ solve the eikonal equations
\begin{align} \label{first eikonal equation}
\begin{split}
\begin{cases}
	\p_t \Xi^{(1)}_t(z') + M_1 z' \cdot \p_{z'} \Xi^{(1)}_t(z') = 0, \ \ (t, z') \in D(0,T) \times \C^{n'}, \\
	\Xi^{(1)}_0 = \Phi^{(1)}_0 \ \textrm{in} \ \C^{n'},
\end{cases}
\end{split}
\end{align}
and
\begin{align} \label{second eikonal equation}
\begin{split}
\begin{cases}
	\p_t \Xi^{(2)}_t(z'') + M_2 z'' \cdot \p_{z''} \Xi^{(2)}_t(z'') = 0, \ \ (t, z'') \in D(0,T) \times \C^{n''}, \\
	\Xi^{(2)}_0 = \Phi^{(2)}_0 \ \textrm{in} \ \C^{n''},
\end{cases}
\end{split}
\end{align}
respectively. The problems (\ref{first eikonal equation}) and (\ref{second eikonal equation}) are globally well-posed in time $t \in \C$, and by inspection we see that their solutions are
\begin{align}
	\Xi^{(1)}_t(z') = \Phi^{(1)}_0 \left(e^{- t M_1} z' \right), \ \ (t,z') \in \C \times \C^{n'}
\end{align}
and
\begin{align}
	\Xi^{(2)}_t(z'') = \Phi^{(2)}_0(e^{- t M_2} z''), \ \ (t,z'') \in \C \times \C^{n''},
\end{align}
respectively. Thus the unique solution to (\ref{quadratic complex time eikonal equation}) is
\begin{align}
	\Xi_t(z) = \Phi^{(1)}_0 \left(e^{- t M_1} z' \right) + \Phi^{(2)}_0(e^{- t M_2} z''), \ \ (t,z) \in D(0,T) \times \C^n.
\end{align}
In view of (\ref{second psh weight}) and (\ref{pullback matrix}), we have
\begin{align} \label{form of evolved quadratic weight}
	\Xi_t(z) = \Phi^{(1)}_0 \left(e^{-t M_1} z' \right) + \frac{1}{2} \sum_{j=1}^{n''} e^{4 \textrm{Re}(t) \epsilon \lambda_j} \abs{z''_j}^2, \ \ (t,z) \in D(0,T) \times \C^n.
\end{align}

Let $0<T_0<T$ be arbitary. We claim that there exists $t_0 \in D(0,T_0) \backslash \{0\}$ such that
\begin{align}
	\Phi_0(z) - \Xi_{t_0}(z) \ge c \abs{z}^2, \ \ z \in \C^{n},
\end{align}
for some $c>0$. Indeed, let us search for such a complex-time $t_0$ of the form
\begin{align} \label{the good complex time}
	t_0 = -\epsilon \rho-i s
\end{align}
where $0< s, \rho \ll 1$. Taking $t = t_0$ in (\ref{form of evolved quadratic weight}) gives
\begin{align} \label{form of the quadratic weight at good time}
	\Xi_{t_0}(z) = \Phi^{(1)}_0 \left(e^{isM_1 + \epsilon \rho M_1} z' \right) + \frac{1}{2} \sum_{j=1}^{n''} e^{-4 \rho \lambda_j} \abs{z''_j}^2, \ \ z \in \C^n.
\end{align}
By Taylor expansion, we have
\begin{align} \label{Taylor expansion weight}
\begin{split}
	\Phi^{(1)}_0 \left(e^{isM_1-\epsilon \rho M_1} z' \right) &= \Xi^{(1)}_{-is} \left(e^{-\epsilon \rho M_1} z' \right) = \Xi^{(1)}_{-is}(z') + \mathcal{O}( \rho \abs{z'}^2), \ \ z' \in \C^{n'}, \ \ 0< s, \rho \ll 1.
\end{split}
\end{align}
From (\ref{first eikonal equation}) and (\ref{pullback q_1}), we see that the $s$-dependent quadratic form $\Xi_{-is}$ is the unique solution of the eikonal equation
\begin{align}
\begin{split}
\begin{cases}
	\p_s \Xi^{(1)}_{-i s}(z') + \textrm{Re} \, \mathfrak{q}_1 \left(z', \frac{2}{i} \p_{z'} \Xi_{-i s}(z') \right) = 0, \ \ (s, z') \in [0, \infty) \times \C^{n'}, \\
	\left. \Xi^{(1)}_{-is} \right|_{s=0} = \Phi^{(1)}_0 \ \textrm{on} \ \C^{n'}.
\end{cases}
\end{split}
\end{align}
Since the quadratic form $q_1$ has trivial singular space, $S_1 = \{0\}$, we know from the results of Section 2 of \cite{SubellipticEstimates} that there is $c>0$ such that
\begin{align} \label{quantitative lower bound for weight}
	\Phi^{(1)}_0(z') - \Xi^{(1)}_{-is}(z') \ge c s^{2 k^{(1)}_0 + 1} \abs{z'}^2, \ \ z' \in \C^{n'}, \ \ 0 \le s \ll 1,
\end{align}
where $k^{(1)}_0$ is the smallest non-negative integer such that
\begin{align}
	\bigcap_{j=0}^{k^{(1)}_0} \textrm{ker} \left[(\textrm{Re} \, F_1) (\textrm{Im} \, F_1)^j \right] \cap \R^{2n'} = \{0\}.
\end{align}
From (\ref{definition of the good weight}), (\ref{second psh weight}), (\ref{form of evolved quadratic weight}), (\ref{form of the quadratic weight at good time}), (\ref{Taylor expansion weight}), and (\ref{quantitative lower bound for weight}), we deduce that there is $c>0$ such that
\begin{align} \label{almost done bound}
	\Phi_0(z) - \Xi_{t_0}(z) \ge c s^{2 k^{(1)}_0 + 1} \abs{z'}^2 + \mathcal{O}(\rho \abs{z'}^2) + \frac{1}{2} \sum_{j=1}^{n''} \left(1 - e^{-4 \rho \lambda_j} \right) \abs{z''_j}^2, \ \ z \in \C^n,
\end{align}
whenever $0<s, \rho \ll 1$ are sufficiently small. Since $\lambda_j > 0$ for all $1 \le j \le n''$, for any $0<s \ll 1$ sufficiently small, we can choose $0< \rho \ll 1$ small enough so that the righthand side of (\ref{almost done bound}) is bounded below by $c \abs{z}^2$ for some $c>0$. Thus, for any $0 < s \ll 1$, there is $0<\rho \ll 1$ and $c>0$ such that
\begin{align} \label{lower bound on weights at good time}
	\Phi_0(z) - \Xi_{t_0}(z) \ge c \abs{z}^2, \ \ z \in \C^n,
\end{align}
for the non-zero complex time $t_0$ given in (\ref{the good complex time}). Taking $s$ and $\rho$ smaller if necessary, we may ensure that $\abs{t_0} < T_0$. From (\ref{quadratic approximation to the evolved weight}) and (\ref{lower bound on weights at good time}), we conclude that there is $c>0$ and $\delta>0$ such that
\begin{align}
	\Phi_{0}(z) - \Phi_{t_0}(z) \ge c \abs{z}^2, \ \ \abs{z} < \delta.
\end{align}

\end{proof}

\begin{remark} \label{tangent space evolution remark}
	An alternative derivation of the quadratic complex-time eikonal equation (\ref{quadratic complex time eikonal equation}) satisfied by $\Xi_t$ may be obtained by considering the tangent spaces $T_0 \Lambda_{\Phi_t}$ for $t \in D(0,T)$. Since $\kappa_t(0) = 0$ for all $t \in D(0,T)$, we have
	\begin{align} \label{image of tangent space}
		T_0 \Lambda_{\Phi_t} = d_0 \kappa_t \left(T_0 \Lambda_{\Phi_0} \right), \ \ t \in D(0,T),
	\end{align}
	where $d_0 \kappa_t$ denotes the differential of $\kappa_t$ at $0 \in \C^{2n}$. Because $\Phi_t$ vanishes to second order at $0 \in \C^n$ for every $t \in D(0,T)$, we have a canonical identification
	\begin{align} \label{canonical identification of tangent space}
		T_0 \Lambda_{\Phi_t} \cong \Lambda_{\Xi_t} , \ \ t \in D(0,T),
	\end{align}
	where $\Lambda_{\Xi_t}$ is the $I$-Lagrangian, $R$-symplectic, subspace of $\C^{2n}$ given by
	\begin{align}
		\Lambda_{\Xi_t} = \set{\left(z, \frac{2}{i} \p_z \Xi_t(z) \right)}{z \in \C^n}, \ \ t \in D(0,T).
	\end{align}
	Moreover, since $\mathfrak{p_0}$ vanishes to second order at $0 \in \C^{2n}$, we have
	\begin{align} \label{identification of differential}
		d_0 \kappa_t = \exp{(tH_{\mathfrak{q}})}, \ \ t \in D(0,T),
	\end{align}
	when we view $d_0 \kappa_t$ as a $\C$-linear transformation $\C^{2n} \rightarrow \C^{2n}$. From (\ref{image of tangent space}), (\ref{canonical identification of tangent space}), and (\ref{identification of differential}), it follows that the family of real quadratic forms $(\Xi_t)_{t \in D(0,T)}$ satisfies
	\begin{align}
	\begin{cases}
		\Lambda_{\Xi_t} = \exp{(tH_{\mathfrak{q}})}(\Lambda_{\Phi_0}), \ \ t \in D(0,T), \\
		\left. \Xi_t \right|_{t=0} = \Phi_0.
	\end{cases}
	\end{align}
	Reasoning similarly to the proof of Proposition \ref{generating function proposition}, we find that $\Xi_t$ solves the quadratic complex-time eikonal equation (\ref{quadratic complex time eikonal equation}).
\end{remark}

We now conclude the proof of Theorem \ref{main_theorem} following the argument sketched in the introduction to this paper. Let $P = \textrm{Op}^w_h(p_0 + h p_1)$ and $u = u(h) \in L^2(\R^n)$ be as in the statement of Theorem \ref{main_theorem}, and let $\varphi$, $\Phi_0$, $U$, $0<T \ll 1$, and $\Phi$ be as in the statement of Proposition \ref{good complex time proposition}. Let $\mathcal{T}_\varphi$ be the FBI transform on $\R^n$ associated to the FBI phase function $\varphi$. By Proposition \ref{locally in evolved spaces}, there are constants $\delta>0$, $0<T_0<T$, $0<C<\infty$, and $0<h_0 \le 1$ such that
\begin{align} \label{bound in the better spaces}
	\sup_{\substack{t \in D(0,T_0) \\ 0<h \le h_0}} \norm{\mathcal{T}_\varphi u}_{L^2_{\Phi_t}(\{\abs{z} < \delta \})} \le C,
\end{align}
where $\norm{\cdot}_{{L^2_{\Phi_t}(\{\abs{z} < \delta \})}}$ is the norm
\begin{align}
	\norm{v}^2_{{L^2_{\Phi_t}(\{\abs{z} < \delta \})}} = \int_{\abs{z}<\delta} \abs{v(z)}^2 e^{-2 \Phi_t(z) /h} \, L(dz).
\end{align}
After taking $\delta>0$ smaller if necessary, we get from Proposition \ref{good complex time proposition} that there is $t_0 \in D(0,T_0) \backslash \{0\}$ and $c>0$ such that
\begin{align} \label{good lower bound for the weights in the end}
	\Phi_0(z) - \Phi^*(z) \ge c \abs{z}^2, \ \ \abs{z} < \delta,
\end{align}
where
\begin{align}
	\Phi^* := \Phi_{t_0}.
\end{align}
Let $N$ be a positive integer that is strictly larger than $n/4$. After taking $0< h_0 \le 1$ smaller if necessary, we obtain from Proposition \ref{first_bound_Lp_norm} that for any $1 \le p \le \infty$ there is $C>0$ such that
\begin{align} \label{second to last inequality}
	\norm{u}_{L^p(\R^n)} \le C h^{\frac{n}{2p}-\frac{3n}{4}} \int_{\abs{z} < \delta} \abs{\mathcal{T}_\varphi u(z)} e^{-\Phi_0(z)/h} \, L(dz) + Ch^N, \ \ 0< h \le h_0.
\end{align}
Thanks to (\ref{bound in the better spaces}), we have
\begin{align} \label{bounded at the good time}
	\sup_{0< h \le h_0} \norm{\mathcal{T}_\varphi u}_{L^2_{\Phi_{t_0}}(\{\abs{z}<\delta \})} < C
\end{align}
for some $0< C < \infty$. From (\ref{good lower bound for the weights in the end}) and (\ref{bounded at the good time}), we may conclude that there is $C>0$ such that
\begin{align} \label{last inequality}
	\int_{\abs{z}<\delta} \abs{\mathcal{T}_\varphi u(z)} e^{-\Phi_0(z)/h} \, L(dz) \le \int_{\abs{z}<\delta} \abs{\mathcal{T}_\varphi u(z)} e^{- \Phi^*(z)/h} e^{-c \abs{z}^2/h} \, L(dz) \le C h^{\frac{n}{2}},
\end{align}
where the second inequality follows from an application of Cauchy-Schwarz and the fact that
\begin{align}
	\left(\int_{\abs{z}<\delta} e^{-c \abs{z}^2 / h} \, L(dz) \right)^{1/2} = \mathcal{O} (h^{\frac{n}{2}}),
\end{align}
which may be deduced from a direct calculation. Putting (\ref{second to last inequality}) and (\ref{last inequality}) together, we find that there is $0< h_0 \le 1$ such that for any $1 \le p \le \infty$ we have
\begin{align}
	\norm{u}_{L^p(\R^n)} \le \mathcal{O}(1) h^{\frac{n}{2p} - \frac{n}{4}} + \mathcal{O}(1) h^N = \mathcal{O}(1)h^{\frac{n}{2p} - \frac{n}{4}}, \ \ 0< h \le h_0.
\end{align}
The proof of Theorem \ref{main_theorem} is complete.


\bibliographystyle{plain}
\bibliography{references}

\end{document}